\documentclass{scrarticle}
\RequirePackage{amsthm, amsmath, amsfonts, amssymb}
\RequirePackage[authoryear]{natbib}
\RequirePackage[colorlinks, citecolor=blue, urlcolor=blue]{hyperref}
\RequirePackage{graphicx, xcolor}
\usepackage[T1]{fontenc}
\usepackage{lmodern}
\usepackage{bbm}
\usepackage{tikz-cd}
\usepackage[margin=2.7cm]{geometry}
\usetikzlibrary{arrows.meta, positioning, calc}
\usepackage{comment}

\theoremstyle{plain}
\newtheorem{axiom}{Axiom}
\newtheorem{claim}[axiom]{Claim}
\newtheorem{theorem}{Theorem}[section]
\newtheorem{lemma}[theorem]{Lemma}
\newtheorem{coro}[theorem]{Corollary}

\theoremstyle{definition}

\newtheorem{example}[theorem]{Example}

\newtheorem*{remark}{Remark}
\newcommand{\R}{\mathbb{R}}

\newcommand{\Prob}{\mathbb{P}}
\newcommand{\E}{\mathbb{E}}
\DeclareMathOperator{\Var}{Var}
\DeclareMathOperator{\Aut}{Aut}
\DeclareMathOperator{\Cay}{Cay}
\DeclareMathOperator{\im}{im}

\DeclareMathOperator{\ext}{ext}
\DeclareMathOperator{\coker}{coker}

\usepackage{authblk}
\begin{document}
  \title{Topology of Percolation Clusters: Central Limit Theorems beyond the Lattice}

\author[1]{Luciano Henrique Lacerda de Araújo}
\author[1]{Daniel Miranda Machado}
\author[1]{Cristian Favio Coletti}

\affil[1]{Center for Mathematics, Computation and Cognition, Universidade Federal do ABC (UFABC)\\
\texttt{henrique.luciano@aluno.ufabc.edu.br}, \texttt{daniel.miranda@ufabc.edu.br}, \texttt{cristian.coletti@ufabc.edu.br}}

\date{} 

\maketitle
\begin{abstract}
  We prove central limit theorems (CLTs) for topological functionals of
  Bernoulli bond percolation on infinite graphs beyond the Euclidean lattice
  $\mathbb{Z}^{d}$. For quasi-transitive graphs of subexponential
  growth, we show that the number $K_{r}$ of open clusters intersecting
  the metric ball $B_{r}$ satisfies a CLT as $r\to\infty$. For
  amenable Cayley graphs, we prove a general CLT for stationary percolation
  functionals along Følner sequences under sequential stabilization and
  a finite-moment assumption, provided the group admits a left-orderable
  finite-index subgroup. This applies in particular to groups of polynomial
  growth. As an application, we obtain CLTs for Betti numbers of graph-generated
  random simplicial complexes, including clique and neighbor complexes.
  The proofs combine invariant edge orderings, martingale
  decompositions, and stabilization estimates for single-edge perturbations.
  
  \vspace{1em}
  \noindent\textbf{MSC2020:} Primary 60K35; Secondary 60F05, 05C80, 20F65, 60D05, 55N35.
  
  \noindent\textbf{Keywords:} Bernoulli percolation, amenable Cayley graphs, central limit theorem, stabilizing functionals, random simplicial complexes.
\end{abstract}

\section{Introduction}

Bernoulli percolation on infinite graphs has long been studied from
probabilistic, geometric, and group-theoretic perspectives. While fluctuation
results are well understood on the lattice $\mathbb{Z}^{d}$, much less is known
for more general infinite graphs, even under strong symmetry assumptions. In
this paper, we prove central limit theorems for large-scale observables of
Bernoulli bond percolation beyond the lattice setting. Our main examples are
cluster counts in large regions and Betti numbers of simplicial complexes
associated with the percolation subgraph.

Our approach combines martingale decompositions built from invariant orderings
of the edge set with stabilization estimates for single-edge perturbations.
This yields a general CLT for stationary percolation functionals along Følner
sequences, and in particular CLTs for cluster counts and homological
observables. The argument applies to amenable Cayley graphs with a
left-orderable finite-index subgroup; in particular, this includes groups of
polynomial growth.

Let $G=(V,E)$ be an infinite, connected, locally finite graph. Fix
$p\in(0,1)$ and let $\omega\in\{0,1\}^{E}$ be an i.i.d.\ percolation
configuration in which each edge is open with probability $p$. For $r\ge 0$,
write $B_{r}$ for the closed ball of radius $r$ centered at a fixed origin $o$.
Let $G(\omega;B_{r})$ be the random subgraph of $G$ restricted to $B_{r}$, with
vertex set $B_{r}$ and edge set
\[
  \{\,\{x,y\}\in E:\ x,y\in B_{r},\ \omega(\{x,y\})=1\,\}.
\]
Let $K_{r}(\omega)$ denote the number of connected components of
$G(\omega;B_{r})$. In addition to $K_{r}$, we consider a broader class of
observables on amenable Cayley graphs: stationary functionals evaluated along
an exhaustive Følner sequence $(A_{r})$ and satisfying a stabilization
property under single-edge perturbations. Finally, given a local rule that
assigns to each finite subgraph $H\subseteq G$ a simplicial complex
$\Delta(H)$, we study the $n$-th Betti numbers $\beta_{n}^{r}$ of
$\Delta(G(\omega;A_{r}))$.

We prove three main results.

\smallskip
\noindent
\textbf{(1) CLT for the number of clusters on quasi-transitive graphs.}
If $G$ is quasi-transitive and of subexponential growth, then the number of
connected components $K_{r}$ of the restricted percolation subgraph
$G(\omega;B_{r})$ satisfies a CLT after centering and normalization by its
variance (Theorem~\ref{teo:perc}). This extends Zhang's martingale method from
$\mathbb{Z}^{d}$ to quasi-transitive graphs of subexponential growth, replacing
lattice translations with averages over large sets.

\smallskip
\noindent
\textbf{(2) A CLT for stabilizing percolation functionals on Cayley graphs.}
Let $G=\Cay(\Gamma,S)$ be an amenable Cayley graph. We prove a CLT for
stationary percolation functionals along an exhaustive Følner sequence under
sequential stabilization and a finite-moment assumption
(Theorem~\ref{TLC_func}). We then show that the required left-orderability
hypothesis holds for every group of polynomial growth
(Theorem~\ref{TLC_funcionais_polinomial}), and hence for virtually nilpotent
Cayley graphs.

\smallskip
\noindent
\textbf{(3) CLTs for Betti numbers of graph-generated random complexes.}
As an application, we obtain CLTs for the Betti numbers of local simplicial
complexes generated from the percolation subgraph, including the clique and
neighbor complexes (Theorem~\ref{TLC_hom}). The main input is a uniform local
bound on the effect of changing a single edge on $\beta_{n}$.

\medskip
\noindent
\textbf{Related work.}
For $\mathbb{Z}^{d}$, asymptotic results and CLTs for the number and size of
percolation clusters were obtained in
\cite{CoxGrimmett1984,KestenZhang1990,SugimineTakei2006}; see also
\cite{GrimmettBook}. For percolation on transitive and Cayley graphs, see
\cite{BenjaminiSchramm1996}. The martingale methods used here go back to
McLeish \cite{mcleish1974dependent} and to the work of Kesten and Zhang
\cite{KestenZhang1997,zhang2001martingale}. Theorem~\ref{teo:perc} extends
this circle of ideas to quasi-transitive graphs of subexponential growth.

Central limit theorems for stabilizing functionals were developed in
stochastic geometry by Penrose \cite{Penrose2001CLT} and in later works such
as \cite{PenrosePeres2011,LRS19}; related results for quasi-local statistics
on Cayley graphs appear in \cite{AVY18}. Our Theorem~\ref{TLC_func} is in the
same spirit, but is tailored to Bernoulli percolation on amenable Cayley
graphs through invariant edge orderings and averaging along Følner sequences.

Limit theorems for topological invariants of random complexes are known in
several Euclidean and mean-field models; see, for example,
\cite{Kahle2009,BobrowskiKahle2018,KM13,YSA17,HST18,KH22,KP25,OT21,OS25}.
Here we treat clique-type complexes generated by Bernoulli percolation on
Cayley graphs and prove CLTs for the Betti numbers $\beta_{n}$ along
increasing windows.

\medskip
\noindent
\textbf{Organization of the paper.}
Section~\ref{sec:prel} introduces the graph-theoretic and probabilistic setup,
including amenability and Følner sequences, the percolation model, and the
local homological constructions we consider. Section~\ref{sec:main} states the
main theorems. The proof of the CLT for the number of clusters
(Theorem~\ref{teo:perc}) is given in Section~\ref{sec:clusters}.
Section~\ref{sec:functionals} establishes the general CLT for stabilizing
functionals on amenable Cayley graphs (Theorem~\ref{TLC_func}) and proves the
polynomial-growth criterion (Theorem~\ref{TLC_funcionais_polinomial}).
Finally, Section~\ref{sec:homology} applies the functional theorem to Betti
numbers and proves Theorem~\ref{TLC_hom}.

  \section{Preliminaries: Graphs, Percolation and Homology}
  \label{sec:prel} In this section, we introduce the graph-theoretic framework,
  the percolation model, and the homological background required for our results.

  \subsection{Graph Structure and Geometry}

  Throughout this paper, $G=(V,E)$ denotes an infinite, connected, locally finite
  graph. We denote the set of neighbors of a vertex $x$ in $G$ by $N_{G}(x)$. We
  also use the shorthand $N(x):=N_{G}(x )$. Given a subset $A \subset V$, let $G
  [A]$ denote the induced subgraph of $G$ with vertex set $A$ and edge set
  \[
    E(A) :=\{\{u,v\} \in E : u,v \in A\}.
  \]

  Fix an origin $o \in V$ and let $d_{G}$ denote the standard graph distance (shortest-path
  metric) on $V$. The closed ball (with respect to $d_{G}$) of radius $r\ge 0$
  centered at $u\in V$ is
  \[
    B_{r}(u) := \{v \in V : d_{G}(u,v) \leq r\},
  \]
  and we write $B_{r}:= B_{r}(o)$ for the ball centered at the origin.

  Our results depend on the asymptotic geometry of the graph. We say that $G$
  has:
  \begin{itemize}
    \item \emph{Polynomial growth} if there exist constants $C>0$ and
      $D\ge 1$ such that
      \[
        |B_{r}(v)| \le C r^{D}\qquad \text{for all }r\ge 1 \text{ and all
        }v\in V.
      \]

    \item \emph{Subexponential growth} if
      \[
        \lim_{r\to\infty}\frac{1}{r}\log |B_{r}(v)| = 0 \qquad \text{for
        all }v\in V.
      \]

    \item \emph{Exponential growth} if
      \[
        \liminf_{r\to\infty}\frac{1}{r}\log |B_{r}(v)| > 0 \qquad \text{for
        all }v\in V.
      \]
  \end{itemize}

  Note that polynomial growth is a strict subcase of subexponential growth. In
  quasi-transitive graphs, these asymptotic growth classes do not depend on
  the choice of the reference vertex $v$, allowing us to fix an origin $o\in V$
  and simply write $B_{r}:= B_{r}(o)$.

  The \emph{$\mathbb{Z}^{d}$ lattice} is a classical example of a graph with
  polynomial growth. Further examples can be found in section ``A. Generalized
  lattices'' in \cite{woess2000random}. \emph{Regular trees} and the Cayley
  graph generated by the \emph{lamplighter group} are examples of graphs with exponential
  growth that are, respectively, non-amenable and amenable; see Section~3.4 on
  Cayley graphs in \cite{lyons2017probability}. In \cite{grigorchuk1985degrees},
  it was proven that the \emph{Grigorchuk group} has subexponential growth but
  grows faster than any polynomial, answering Milnor's question
  \cite{milnor1968} on the existence of groups with such growth
  characteristics.

  For $W \subset V$, we define the (inner) \emph{vertex boundary} and the
  \emph{edge boundary}, respectively, by
  \[
    \partial_{V}W := \{\, u \in W : \exists v \in V \setminus W \text{ with }
    \{u,v\} \in E \,\}
  \]
  and
  \[
    \partial_{E}W := \{\, \{u,v\} \in E : u \in W,\ v \in V \setminus W \,\}.
  \]
  Unless otherwise specified, we write $\partial W$ as shorthand for
  $\partial_{V}W$.

  The \emph{isoperimetric constant} (or \emph{Cheeger constant}) of $G$ is
  given by
  \[
    \Phi(G) := \inf \left\{ \frac{|\partial W|}{|W|}: \varnothing \neq W \subset
    V,\ |W| < \infty \right\}.
  \]
  The graph $G$ is \emph{amenable} if $\Phi(G) = 0$, and \emph{non-amenable} otherwise.
  In bounded-degree quasi-transitive graphs, $|\partial_{V}W|$ and $|\partial_{E}
  W|$ are comparable up to constants, uniformly over finite $W$: there exist constants
  $c_{1},c_{2}>0$ such that
  \[
    c_{1}\,|\partial_{E}W| \le |\partial_{V}W| \le c_{2}\,|\partial_{E}W| \qquad
    \text{for all finite }W \subset V,
  \]
  so either boundary notion can be used in the definition of $\Phi(G)$.

  A classical characterization states that $G$ is amenable if and only if it admits
  a \emph{F{\o}lner sequence}: a sequence of finite, non-empty subsets
  $(A_{r})_{r \geq 1}$ of $V$ such that
  \begin{equation*}
    \lim_{r \to \infty}\frac{|\partial A_{r}|}{|A_{r}|}= 0.
  \end{equation*}
  We say that the F{\o}lner sequence $(A_{r})_{r \geq 1}$ is \emph{exhaustive}
  if $A_{r}\subseteq A_{r+1}$ for all $r \geq 1$ and $\bigcup_{r=1}^{\infty}A_{r}
  = V$.

  \subsection{Symmetries and Cayley Graphs}
  The algebraic structure of Cayley graphs is essential for our later results,
  while Theorem \ref{teo:perc} holds in the more general quasi-transitive
  setting. Let $\Gamma$ be a finitely generated group. We say that $G = (V,E)$
  is a \emph{Cayley graph} of $\Gamma$ if there exists a fixed, finite, symmetric
  generating set $S$ (where $S = S^{-1}$ and $1 \notin S$) of $\Gamma$ such that
  $V=\Gamma$, and the edge set is defined by the group action: $E := \{\{g,gs\}
  : g \in \Gamma, s \in S\}$. In this case, we denote $G = \Cay(\Gamma, S)$.
  Throughout this paper, we assume $\Gamma$ to be infinite.

  In the context of Cayley graphs, amenability can be equivalently
  characterized by the algebraic F{\o}lner condition. Specifically, a finitely
  generated group $\Gamma$ is amenable (and hence, so is any Cayley graph
  generated by it) if and only if there exists a sequence of finite, non-empty
  subsets $(A_{r})_{r \geq 1}$ of $\Gamma$ such that for every $g \in \Gamma$,
  \[
    \lim_{r \to \infty}\frac{|gA_{r}\mathbin{\triangle}A_{r}|}{|A_{r}|}= 0,
  \]
  where $\mathbin{\triangle}$ denotes the symmetric difference.

  The symmetry of a Cayley graph is generalized to arbitrary graphs via their automorphism
  groups. An automorphism of a graph $G$ is an adjacency-preserving bijection
  $\varphi: V \to V$. The set of all such bijections forms the \emph{automorphism
  group}, denoted by $\Aut(G)$. The orbit of a vertex $v \in V$ under this
  group is $\mathcal{O}_{v}:= \{\varphi(v) :\varphi \in \Aut(G)\}$. We say that
  $G$ is \emph{vertex-transitive} if it has a single orbit, and \emph{quasi-transitive}
  if the action of $\Aut(G)$ on $V$ has finitely many orbits. Throughout this paper,
  we restrict our attention to graphs that are transitive or quasi-transitive.
  If $G$ is quasi-transitive, then the growth rate does not depend on the
  choice of the reference vertex; hence we measure growth from the fixed origin
  $o$.

  Given an action of a group $H$ on $E$, a subset $\tilde{E}\subseteq E$ is
  called an \emph{$H$-fundamental set of edges} for this action if for every $e
  \in E$, there exists $h \in H$ such that $h\cdot e \in \tilde{E}$. Note that
  this $h$ need not be unique: as we shall see, depending on the choice of $H$
  and $\tilde{E}$, an edge $e \in E$ might be mapped into $\tilde{E}$ by both
  $h$ and $h^{-1}$.

  For a Cayley graph, given $v \in V$, let $\varphi_{v}:V \to V$ denote the
  left-translation automorphism $\varphi_{v}(u) = v\cdot u$. Given a F{\o}lner
  sequence $(A_{r})_{r \geq 1}$ in $V$, we call the \emph{F{\o}lner sequence
  orbit} the set
  \[
    \mathcal{A}:= \{vA_{r}: v \in V, r \geq 1\},
  \]
  where $vA_{r}= \{\varphi_{v}(u) : u \in A_{r}\}$.

  Finally, we consider groups admitting a particular invariant structure. A group
  $\Gamma$ is \emph{left-orderable} (LO) if it admits a strict total ordering
  $<$ that is left-invariant: $g < h \implies kg < kh$ for all $g,h,k \in \Gamma$.
  For a comprehensive treatment of the properties and characterizations of LO
  groups, we refer the reader to \cite{clay2023orderable}.

  \subsection{Percolation and Functionals}
  We consider Bernoulli bond percolation on $G$ with parameter $p \in (0,1)$.
  Let $\Omega := \{0,1\}^{E}$ be the configuration space, equipped with the product
  measure $\Prob_{p}$ under which edges are declared ``open'' ($\omega( e) =1$)
  independently with probability $p$. We denote by $\E_{p}$ and $\Var_{p}$ the
  corresponding expectation and variance.

  Let $E(\omega) := \{e \in E: \omega(e) = 1\}$ be the set of open edges, which
  defines the random subgraph $G(\omega) := (V,E(\omega))$. The connected
  component of a vertex $v$ in the configuration $\omega$ is called the \emph{cluster}
  of $v$, denoted by $\mathcal{C}(\omega; v)$. When $\omega$ is clear, we
  write $\mathcal{C}(v)$. The \emph{phase transition} of the model is
  characterized by the \emph{critical parameter}:
  \[
    p_{c}:=\sup\bigl\{p\in[0,1]:\Prob_{p}\bigl(|\mathcal{C}(\omega;o)|=\infty
    \bigr)=0\bigr\}.
  \]
  If $G$ is quasi-transitive, this definition does not depend on the choice of
  $o$.

  The process is said to be in the \emph{subcritical phase} if $p < p_{c}$ and
  in the \emph{supercritical phase} if $p > p_{c}$.

  Given a subset $A \subseteq V$, we denote by $G(\omega;A)$ the random
  subgraph $(A, E(\omega)\cap E(A))$, and by $\mathcal{C}_{A}(\omega;v)$ the cluster
  of $v$ restricted to $G(\omega;A)$. We let $K_{r}(\omega)$ denote the number
  of connected components of the restricted random subgraph $G(\omega;B_{r})$.

  While the number of clusters is a fundamental quantity, our analysis extends
  to a broader class of geometric observables in amenable Cayley graphs. Let
  $( A_{r})_{r\geq 1}$ be a F{\o}lner sequence and let $\mathcal{A}$ be its
  orbit. An \emph{$\mathcal{A}$-functional} is a measurable map
  $F:\Omega\times\mathcal{A}\to\mathbb{R}$. When $\omega$ is clear from the context,
  we write $F(A):=F(\omega,A)$.

  We require the functional $F$ to be compatible with the underlying
  symmetries and to satisfy stabilization assumptions under local
  perturbations. The automorphism $\varphi_{u}$ acts on edges by
  \[
    \varphi_{u}(\{x,y\})=\{\varphi_{u}(x),\varphi_{u}(y)\}.
  \]
  We define the induced action on configurations by
  \[
    (\tilde\varphi_{u}\omega)(e):=\omega(\varphi_{u}^{-1}e),\qquad e\in E.
  \]
  That is, the map $\tilde\varphi_{u}:\Omega\to\Omega$ translates the
  configuration $\omega$ by the automorphism $\varphi_{u}$.

  An $\mathcal{A}$-functional $F$ is called \emph{stationary} if it is invariant
  under the diagonal action of any translation; that is, for any $u \in \Gamma$
  and $A \in \mathcal{A}$,
  \[
    F(\omega, A) = F(\tilde{\varphi}_{u}(\omega), \varphi_{u}(A)) \quad \Prob
    _{p}\text{-a.s.}
  \]

  We quantify the impact of a single edge on the functional via the difference operators.
  For an edge $e\in E$, we set
  \[
    D_{e}(\omega,A):=F(\omega,A)-F(\omega^{e},A),
  \]
  where $\omega^{e}$ is obtained from $\omega$ by replacing the state of $e$ with
  an independent copy; that is, $\omega^{e}(e)$ is independent of $\omega(e)$,
  and $\omega^{e}(e')=\omega(e')$ for $e'\neq e$. We say that an $\mathcal{A}$-functional
  $F$ \emph{stabilizes in sequence} in $K \subseteq E$ if, for any edge $e \in
  K$, there exists a random variable $D^{\infty}_{e}$ such that, for any $x \in
  V$, the sequence of random variables $D_{e}(\omega , xA_{r})$ converges in
  probability to a limit $D^{\infty}_{e}(\omega)$ as $r \to \infty$. In particular,
  this limit does not depend on $x$.

  We say that $F$ satisfies the \emph{finite moment condition} in $K$ if there
  exists $\gamma >2$ such that
  \[
    \max_{e\in K}\ \sup_{A\in\mathcal{A}}\ \E_{p}\bigl[|D_{e}(\omega,A)|^{\gamma}
    \bigr]<\infty.
  \]

  The number of clusters, the number of isolated vertices, and the total
  perimeter of clusters are three examples of stationary $\mathcal{A}$-functionals
  in this setting. These functionals are clearly stationary by definition.
  Moreover, they satisfy the stabilization and moment conditions, since
  the effect of a single-edge perturbation on the functional is localized within a random, finite radius.

  \subsection{Simplicial Homology of Graphs}
  A primary application of this functional framework is the study of random
  graph homology. A \emph{simplicial complex} on a vertex set $V$ is a
  collection $\Delta \subseteq \mathcal{P}(V)$, the power set of $V$, such
  that if $A \in \Delta$ and $B \subseteq A$, then $B \in \Delta$. The
  elements of $\Delta$ are called \emph{simplices}. If $\sigma$ is a simplex, its
  \emph{dimension} is defined as $\dim \sigma = |\sigma| - 1$. A simplex of dimension
  $n$ is referred to as an \emph{$n$-simplex}. Given a simplicial complex $\Delta
  = \{\sigma_{\alpha}\}_{\alpha \in \Lambda}$, a \emph{simplicial subcomplex} is
  a sub-collection
  \[
    \Delta^{\prime}= \{\sigma_{\alpha}\}_{\alpha \in \mathcal{B}},
  \]
  where $\mathcal{B}\subseteq \Lambda$ is such that $\Delta^{\prime}$ itself satisfies
  the structure of a simplicial complex.

  Two simplicial complexes $\Delta$ and $\Delta^{\prime}$ are said to be \emph{isomorphic}
  if there exists a bijection $\varphi: \Delta \to \Delta^{\prime}$ that
  preserves both inclusion and dimension; that is, for any
  $\sigma, \rho \in \Delta$, $\sigma \subseteq \rho \iff \varphi(\sigma) \subseteq
  \varphi(\rho)$, and for every $\sigma \in \Delta$, $\dim \sigma = \dim \varphi
  (\sigma)$. In this case, we denote the isomorphism by $\Delta \simeq \Delta^{\prime}$.

  Given a graph $G = (V, E)$, a \emph{graph simplicial complex} on $G$ is a mapping
  $\Delta$ that assigns to each subgraph $H = (W, F)$ a simplicial complex $\Delta
  (H) \subseteq \mathcal{P}(W)$ such that for any $\varphi\in\Aut(G)$ and any subgraph
  $H\subseteq G$,
  \[
    \Delta(\varphi(H))=\varphi(\Delta(H)),
  \]
  where $\varphi(\Delta(H)):=\{\varphi(\sigma):\sigma\in\Delta(H)\}$ and
  $\varphi (\sigma):=\{\varphi(v):v\in\sigma\}$. (\emph{Equivariance property}.)
  To illustrate, we give some examples:

  \begin{example}
    A set of vertices $\sigma = \{v_{0}, \dots, v_{n}\}$ is a \emph{simplex
    of cliques} of $G$ if $\{v_{i}, v_{j}\} \in E$ for all distinct
    $i, j \in \{0, \dots, n\}$. The \emph{clique simplicial complex}
    $\Delta(G)$ is the collection of all such simplices. Figure \ref{fig:simplex_cliques_graph}
    shows examples of simplices of cliques.
  \end{example}

  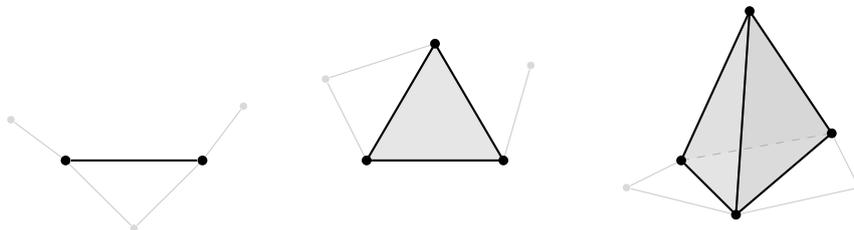
\begin{figure}[htbp]
    \centering
    \begin{tikzpicture}[
      scale=1.8,
      main node/.style={circle, fill=black, inner sep=1.3pt},
      ext node/.style={circle, fill=gray!30, inner sep=1pt},
      main edge/.style={thick},
      graph edge/.style={thin, gray!40},
      hidden edge/.style={dashed, gray!60, thin}
    ]
      \begin{scope}[xshift=0cm]
        \node[ext node] (E1) at (-0.4, 0.3) {};
        \node[ext node] (E2) at (0.5, -0.5) {};
        \node[ext node] (E3) at (1.3, 0.4) {};

        \draw[graph edge] (E1) -- (0,0);
        \draw[graph edge] (E2) -- (0,0);
        \draw[graph edge] (E2) -- (1,0);
        \draw[graph edge] (E3) -- (1,0);

        \node[main node] (A1) at (0,0) {};
        \node[main node] (A2) at (1,0) {};
        \draw[main edge] (A1) -- (A2);
      \end{scope}

      \begin{scope}[xshift=2.2cm]
        \coordinate (B1) at (0,0);
        \coordinate (B2) at (1,0);
        \coordinate (B3) at (0.5,0.86);

        \node[ext node] (E4) at (-0.3, 0.6) {};
        \node[ext node] (E5) at (1.2, 0.7) {};

        \draw[graph edge] (E4) -- (B1);
        \draw[graph edge] (E4) -- (B3);
        \draw[graph edge] (E5) -- (B2);

        \fill[gray!20] (B1) -- (B2) -- (B3) -- cycle;
        \draw[main edge] (B1) -- (B2) -- (B3) -- cycle;
        \node[main node] at (B1) {};
        \node[main node] at (B2) {};
        \node[main node] at (B3) {};
      \end{scope}

      \begin{scope}[xshift=4.5cm]
        \coordinate (C1) at (0,0);
        \coordinate (C2) at (1.1,0.2);
        \coordinate (C3) at (0.5,1.1);
        \coordinate (C4) at (0.4,-0.4);

        \node[ext node] (E6) at (-0.4, -0.2) {};
        \node[ext node] (E7) at (1.3, -0.2) {};

        \draw[graph edge] (E6) -- (C1);
        \draw[graph edge] (E6) -- (C4);
        \draw[graph edge] (E7) -- (C2);
        \draw[graph edge] (E7) -- (C4);

        \fill[gray!10] (C1) -- (C2) -- (C3) -- cycle;
        \fill[gray!30, opacity=0.7] (C1) -- (C3) -- (C4) -- cycle;
        \fill[gray!40, opacity=0.7] (C2) -- (C3) -- (C4) -- cycle;

        \draw[hidden edge] (C1) -- (C2);
        \draw[main edge] (C1) -- (C3);
        \draw[main edge] (C1) -- (C4);
        \draw[main edge] (C2) -- (C3);
        \draw[main edge] (C2) -- (C4);
        \draw[main edge] (C3) -- (C4);

        \node[main node] at (C1) {};
        \node[main node] at (C2) {};
        \node[main node] at (C3) {};
        \node[main node] at (C4) {};
      \end{scope}
    \end{tikzpicture}
    \caption{Realization of a $1$-simplex, a $2$-simplex and a $3$-simplex of
    cliques highlighted in a graph.}
    \label{fig:simplex_cliques_graph}
  \end{figure}

  \begin{example}
    We say that $\sigma$ is a \emph{neighbor simplex} of the graph $G$ if
    $\sigma \subseteq \{v\} \cup N_{G}(v)$, for some $v \in V$. The \emph{neighbor
    simplicial complex} of $G$ is then defined as the collection of all such
    neighbor simplices in $G$.
  \end{example}

  \begin{example}[Weighted Vietoris--Rips complexes]
    Let $G=(V,E)$ be a locally finite, quasi-transitive graph with
    automorphism group $\Gamma\leq \Aut(G)$. Since $\Gamma$ has finitely many
    edge-orbits, any assignment of positive values to these orbits induces a
    $\Gamma$-invariant weight function $w\colon E\to(0,\infty)$. For a subgraph
    $H\subseteq G$, let $d_{H,w}$ denote the weighted path distance on
    $V(H)$,
    \[
      d_{H,w}(u,v):=\inf\Big\{\sum_{e\in P}w(e): P \text{ a path in }H \text{
      from }u \text{ to }v\Big\}.
    \]
    Fix $r>0$. The \emph{(weighted) Vietoris--Rips complex} $\Delta_{r}(H)$ is
    the simplicial complex on $V(H)$ whose simplices are the finite sets $\sigma
    \subseteq V(H)$ such that $d_{H,w}(u,v)\le r$ for all $u,v\in \sigma$.
  \end{example}

  \begin{example}[Weighted graph-\v{C}ech complexes]
    In the setting of Example~2.3, fix $r>0$. The \emph{(weighted) graph-\v{C}ech
    complex} $\check\Delta_{r}(H)$ is the simplicial complex on $V(H)$ whose
    simplices are the finite sets $\sigma\subseteq V(H)$ for which there
    exists $z\in V(H)$ such that $d_{H,w}(z,x)\le r$ for all $x\in \sigma$.
  \end{example}

  \begin{example}[Plaquette complexes]
    Let $G$ be a Cayley graph with a $\Gamma$-invariant collection $\mathcal{P}$
    of cycles (plaquettes) of length at most $k$. For any subgraph
    $H\subseteq G$, the \emph{plaquette complex} $\Delta_{\mathcal{P}}(H)$
    is the simplicial complex generated by the vertex sets of all subgraphs
    of $H$ that are isomorphic to an element of $\mathcal{P}$.
  \end{example}

  \begin{example}[Path complexes]
    Fix $k\ge 1$. For a subgraph $H\subseteq G$, the \emph{path complex}
    $\Delta^{\mathrm{path}}_{k}(H)$ is the simplicial complex generated by
    the vertex sets of simple paths in $H$ of length at most $k$.
  \end{example}

  We say that $\Delta$ is a \emph{local graph simplicial complex rule} over
  $G$ if it assigns to each subgraph $H\subseteq G$ a simplicial complex
  $\Delta(H)$ satisfying:

  \begin{itemize}
    \item \emph{Monotonicity:} If $H\subseteq K$ are subgraphs of $G$, then $\Delta
      (H)\subseteq \Delta(K)$.

    \item \emph{Connectivity:} For every subgraph $H\subseteq G$ and every
      simplex $\sigma\in\Delta(H)$, all vertices of $\sigma$ lie in a single
      connected component of $H$.

    \item \emph{Confinement:} There exists $T\ge0$ such that for every
      subgraph $H\subseteq G$ and every simplex $\sigma\in\Delta(H)$, there
      exists a non-empty subgraph $H_{\sigma}\subseteq H$ with
      \[
        \sup_{u,v\in V(H_\sigma)}d_{G}(u,v)\le T \quad\text{and}\quad \sigma
        \in\Delta(H_{\sigma}).
      \]
  \end{itemize}
  The smallest $T$ with this property is called the \emph{basic diameter} of the
  simplicial complex. Any subgraph $H$ that minimizes
  $\sup_{u,v \in H}d_{G}(u,v)$ among all subgraphs satisfying
  $\sigma \in \Delta(H)$ is called a minimal subgraph associated with $\sigma$.
  Throughout this paper, we restrict our focus to local simplicial complexes. Note
  that the examples of simplicial complexes presented above satisfy the locality
  requirements.

  An \emph{ordered simplex} is a simplex with a fixed ordering of its vertices.
  Two orderings are considered equivalent if one can be obtained from the other
  by an even number of permutations. This equivalence relation defines two
  classes, referred to as the \emph{orientations} of a simplex. A negative sign
  preceding a simplex indicates the reversal of its orientation. Henceforth, oriented
  simplices will be referred to simply as simplices, and the notation
  $[v_{0}, \dots, v_{n}]$ denotes the simplex $\{v_{0}, \dots, v_{n}\}$
  equipped with an orientation.

  Given a finite subgraph $K$ of $G$, let $C_{n}(K)$ be the set of $n$-simplices
  in $\Delta(K)$, and let $S_{n}(K)$ denote the $n$-th \emph{chain space},
  defined as the $\mathbb{R}$-linear span of the oriented $n$-simplices
  associated with $C_{n}(K)$. If $K$ contains no $n$-simplices, $S_{n}(K)$ is
  defined as the trivial vector space. For our purposes, we will define
  homology with respect to a fixed sequence of finite subsets of $V$, denoted $(
  A_{r})_{r \geq 1}$. When the sequence in question is clear from context, we will
  simplify the notation to $C_{n}^{r}(G)$ or $C_{n}^{r}$ to indicate $C_{n}(G[A
  _{r}])$, and $S_{n}^{r}(G)$ or $S_{n}^{r}$ to indicate $S_{n}(G[A_{r}])$.

  The \emph{boundary operator} is defined in the standard way: the $n$-th boundary
  homomorphism $\partial_{n}^{r}\colon S_{n}^{r}\to S_{n-1}^{r}$ acts on oriented
  simplices $[v_{0}, \dots, v_{n}]$ via
  \[
    \partial_{n}^{r}([v_{0}, \dots, v_{n}]) := \sum_{i=0}^{n}(-1)^{i}[v_{0},
    \dots, \widehat{v}_{i}, \dots, v_{n}],
  \]
  and extended linearly, where the hat symbol $\widehat{v}_{i}$ denotes the
  omission of the $i$-th vertex. The $n$-th \emph{simplicial homology group}
  is defined as the quotient space
  \[
    H_{n}^{r}:= \frac{\ker \partial^{r}_{n}}{\operatorname{im}\partial^{r}_{n+1}}
    ,
  \]
  and the $n$-th \emph{Betti number} is its dimension,
  $\beta_{n}^{r}:= \dim H_{n}^{r}$. An element of $H_{n}^{r}$ is a \emph{homology
  class}, and any cycle $z \in \ker \partial_{n}^{r}$ representing such a class
  is called a \emph{homology representative}. We say that a chain $\sum_{j \in
  J}\rho_{j}\sigma_{j}\in S_{n}^{r}(G)$ is a \emph{connected chain} if, for any
  $k, l \in J$, the simplices $\sigma_{k}$ and $\sigma_{l}$ belong to the same
  connected component of the graph. If such a chain constitutes a homology generator,
  it is referred to as a \emph{connected homology generator}.

  Given finite subgraphs $K$ and $L$ of $G$ such that $K \subseteq L$, the
  inclusion $i: K \hookrightarrow L$ induces a homomorphism in homology. By the
  monotonicity property, $\Delta(K) \subseteq \Delta(L)$, which implies that $S
  _{n}(K)$ is a vector subspace of $S_{n}(L)$ and that the inclusion commutes with
  the boundary operators. We define the \emph{induced homomorphism in homology}
  $i_{*}^{n}: H_{n}(K) \to H_{n}(L)$ by
  \[
    z + \partial_{n+1}(S_{n+1}(K)) \mapsto z + \partial_{n+1}(S_{n+1}(L)),
  \]
  where $z \in \ker \partial_{n}(K)$ is an $n$-cycle. This induced
  homomorphism is well-defined. Indeed, the monotonicity property guarantees
  that $\partial_{n+1}(S_{n+1}(K)) \subseteq \partial_{n+1}(S_{n+1}(L))$. So,
  if two cycles represent the same homology class in $H_{n}(K)$, their difference
  is a boundary in $K$. Since this boundary is the image of an element in
  $S_{n+1}(K)$, it is also the image of that same element in $S_{n+1}(L)$,
  ensuring that the mapping is independent of the choice of representative.

  \section{Main Results}
  \label{sec:main} In the following theorems, $\overset{d}{\Longrightarrow}$ denotes
  convergence in distribution. Our first result adapts the martingale method
  of \cite{zhang2001martingale}, which established a Central Limit Theorem (CLT)
  for $K_{r}$ in $\mathbb{Z}^{d}$.

  \begin{theorem}
    \label{teo:perc} Let $G$ be an infinite, connected, locally finite, and
    quasi-transitive graph of subexponential growth. Consider Bernoulli bond
    percolation with parameter $p \in(0,1)$. Then, as $r \to \infty$,
    \[
      \frac{K_{r}- \E_{p}[K_{r}]}{\sqrt{\Var_{p}(K_{r})}}\overset{d}{\Longrightarrow}
      \mathcal{N}(0,1).
    \]
  \end{theorem}

  Although the number of clusters in Theorem \ref{teo:perc} can be analyzed
  using a geometrically defined martingale, our approach to establishing a general
  CLT for stabilizing functionals relies on computing the asymptotic variance via
  the pointwise ergodic theorem. To ensure the strictly translation-invariant
  martingale structure needed for this method, we restrict our focus in our second
  result (extending \cite{Penrose2001CLT}) to amenable Cayley graphs admitting
  a finite-index left-orderable (LO) subgroup $H$.

  \begin{theorem}
    \label{TLC_func} Let $G$ be the Cayley graph of an amenable group $\Gamma$
    admitting a finite-index left-orderable subgroup $H$, with
    $n=[\Gamma:H]$. Consider Bernoulli bond percolation with parameter $p\in(
    0,1)$, an exhaustive F{\o}lner sequence $(A_{r})$, and let $\mathcal{A}$
    be its orbit under $\Gamma$. Then there exists an $H$-fundamental set of
    edges $\tilde E\subset E$ such that, if the stationary $\mathcal{A}$-functional
    $F$ stabilizes in sequence and satisfies a finite moment condition on $\tilde
    E$, then:
    \begin{enumerate}
      \item $\displaystyle \lim_{r\to\infty}\frac{1}{|A_{r}|}\Var_{p}\!\bigl
        (F(A_{r})\bigr)=\sigma^{2}$;

      \item $\displaystyle |A_{r}|^{-1/2}\bigl(F(A_{r})-\E_{p}[F(A_{r})]\bigr
        )\ \overset{d}{\Longrightarrow}\ \mathcal{N}(0,\sigma^{2})$,
    \end{enumerate}
    where
    \[
      \sigma^{2}= \frac{1}{2n}\sum_{e\in\tilde E}\E_{p}\!\left[\E_{p}\!\left
      [D_{e}^{\infty}\,\middle|\,\mathcal{F}_{e}\right]^{2}\right].
    \]
    Here $D_{e}^{\infty}$ is the limit in probability of $D_{e}(\omega,xA_{r}
    ):=F(\omega,xA_{r})-F(\omega^{e},xA_{r})$ as $r\to\infty$ (the limit
    does not depend on $x$), and $\mathcal{F}_{e}$ is the $\sigma$-algebra
    generated by the edges preceding or equal to $e$ in an $H$-invariant total order on
    $E$.
  \end{theorem}

  \begin{remark}
    An interesting observation is that, as we shall see in the proof, the random
    variable $|A_{r}|^{-1/2}\bigl(F(A_{r})-\E_{p}[F(A_{r})]\bigr)$ is decomposed
    into a sum over the edges in $A_{r}$, the total number of which we
    denote by $b_{r}$. If we instead normalize by $b_{r}^{1/2}$ and consider
    the random variable $b_{r}^{-1/2}\bigl(F(A_{r})-\E_{p}[F(A_{r})]\bigr)$,
    its asymptotic variance is given by
    \[
      \sigma^{2}=\frac{1}{n|S|}\sum_{e\in\tilde E}\E_{p}\!\left[\E_{p}\!\left
      [D_{e}^{\infty}\,\middle|\,\mathcal{F}_{e}\right]^{2}\right],
    \]
    which, given the fundamental set of edges $\tilde{E}$ to be introduced later,
    yields the average over the fundamental edges:
    \[
      \sigma^{2}=\frac{1}{|\tilde{E}|}\sum_{e\in\tilde E}\E_{p}\!\left[\E_{p}
      \!\left[D_{e}^{\infty}\,\middle|\,\mathcal{F}_{e}\right]^{2}\right].
    \]
    An analogous statement holds for Theorem \ref{TLC_hom}.
  \end{remark}

  While Theorem \ref{TLC_func} requires the existence of an LO subgroup, this
  condition is naturally satisfied in the presence of polynomial growth, as established
  below.

  \begin{theorem}
    \label{TLC_funcionais_polinomial} Let $\Gamma$ be an infinite, finitely
    generated group of polynomial growth. Then Theorem~\ref{TLC_func} applies
    to any Cayley graph $G=\Cay(\Gamma,S)$ associated with a finite symmetric
    generating set $S$.
  \end{theorem}

  As an application of Theorem \ref{TLC_func}, we establish a Central Limit Theorem
  for the Betti numbers of the associated random simplicial complexes.
  \begin{theorem}
    \label{TLC_hom} Let $G$ be the Cayley graph of an amenable group $\Gamma$
    admitting a finite-index left-orderable subgroup $H$, with
    $n=[\Gamma:H]$. Consider Bernoulli bond percolation on $E(G)$ with parameter
    $p\in(0,1)$ and an exhaustive F{\o}lner sequence $(A_{r})$. Fix a local
    graph simplicial complex rule $\Delta$, and let $\beta_{n}^{r}$ denote the
    $n$-th Betti number of the simplicial complex $\Delta(G(\omega;A_{r}))$.

    Let $\preccurlyeq$ and $(\mathcal{F}_{e})_{e\in E(G)}$ be as in Theorem~\ref{TLC_func}.
    If $\beta_{n}^{r}>0$ a.s.\ for some $r$, then there exists a finite $H$-fundamental
    set of edges $\tilde{E}\subset E(G)$ such that:
    \begin{enumerate}
      \item $\displaystyle \lim_{r\to\infty}\frac{1}{|A_{r}|}\Var_{p}(\beta
        _{n}^{r})=\sigma^{2}$;

      \item $\displaystyle |A_{r}|^{-1/2}\Bigl(\beta_{n}^{r}-\E_{p}[\beta_{n}
        ^{r}]\Bigr)\ \overset{d}{\Longrightarrow}\ \mathcal{N}(0,\sigma^{2}
        )$.
    \end{enumerate}
    Moreover, the asymptotic variance is given by
    \[
      \sigma^{2}=\frac{1}{2n}\sum_{e\in\tilde{E}}\E_{p}\!\left[\E_{p}\!\left
      ( D_{e}^{\infty}\,\middle|\, \mathcal{F}_{e}\right)^{2}\right ],
    \]
    where $D_{e}^{\infty}$ is the limit in probability, as $r\to\infty$, of
    the difference in $\beta_{n}^{r}$ caused by perturbing the state of $e$.
  \end{theorem}

  \begin{remark}
    For $n=0$, the Betti number $\beta_{0}^{r}$ counts the number of
    connected components of the simplicial complex $\Delta(G(\omega;A_{r}))$.
    Under the connectivity assumption on $\Delta$, this coincides with the
    number of connected components of the underlying percolation subgraph. Thus,
    Theorem \ref{TLC_hom} provides an analogue of Theorem \ref{teo:perc} for
    F{\o}lner sequences, albeit under more restrictive algebraic assumptions
    on the graph (requiring a Cayley graph with an LO subgroup, rather than mere
    quasi-transitivity and subexponential growth).
  \end{remark}

  \section{Central Limit Theorem for the Number of Clusters}
  \label{sec:clusters}

  \begin{proof}[Proof of Theorem \ref{teo:perc}]
    Fix a total ordering $e_{1}, e_{2}, \dots$ of the edge set $E$ such that
    for all $r \geq 1$, $E_{r}:= E(B_{r}) = \{e_{1}, \dots, e_{b_r}\}$, and for
    any $j \geq 1$, the subgraph induced by $\{e_{1}, \dots, e_{j}\}$ on its
    vertex set $V_{j}$ is connected. Write $\omega_{j}:= \omega(e_{j})$. We
    define the filtration $\{\mathcal{F}_{j}\}_{j \geq 0}$, where $\mathcal{F}
    _{0}:= \{\varnothing,\Omega\}$ is the trivial $\sigma$-algebra and $\mathcal{F}
    _{j}:= \sigma(\omega_{1}, \dots, \omega_{j})$ for $j \geq 1$.

    For $r \geq 1$, define the sequence $\{\delta_{r,j}\}_{j=1}^{b_r}$ by
    \[
      \delta_{r,j}:= \E_{p}[K_{r}\mid\mathcal{F}_{j}] - \E_{p}[K_{r}\mid\mathcal{F}
      _{j-1}].
    \]
    This is a martingale difference sequence with respect to
    $(\Omega, \mathcal{F}, \{\mathcal{F}_{j}\}_{j\geq 0}, \Prob_{p})$. Since
    $K_{r}$ is $\mathcal{F}_{b_r}$-measurable, we have the telescoping sum
    \[
      K_{r}- \E_{p}[K_{r}] = \sum_{j=1}^{b_r}\delta_{r,j}.
    \]
    Let $X_{r,j}:= \delta_{r,j}/ \sqrt{\Var_{p}(K_{r})}$. The collection $\{X
    _{r,j}\}_{r \geq 1, 1 \leq j \leq b_r}$ forms a martingale difference array,
    and
    \[
      \frac{K_{r}- \E_{p}[K_{r}]}{\sqrt{\Var_{p}(K_{r})}}= \sum_{j=1}^{b_r}
      X_{r,j}.
    \]
    We apply the following Martingale Central Limit Theorem (cf.\cite{mcleish1974dependent}).

    \begin{theorem}[McLeish, 1974]
      \label{convergencia_martingais_TLC} Let $\{X_{r,j}\}_{r \geq 1, 1 \leq
      j \leq b_r}$ be a martingale difference array with respect to filtrations
      $\{\mathcal{F}_{r,j}\}$, satisfying:
      \begin{enumerate}
        \item $\max_{1 \leq j \leq b_r}|X_{r,j}|$ is uniformly bounded in
          $L_{2}$-norm;

        \item $\max_{1 \leq j \leq b_r}|X_{r,j}|$ converges in probability
          to $0$; and

        \item $\sum_{j=1}^{b_r}X^{2}_{r,j}$ converges in probability to $1$.
      \end{enumerate}
      Then, denoting $S_{r}= \sum_{j=1}^{b_r}X_{r,j}$, $S_{r}$ converges in
      distribution to $\mathcal{N}(0,1)$.
    \end{theorem}

    To verify conditions (1) and (2), we first bound $\delta_{r,j}$. Note that,
    by the Law of Total Expectation, we have
    \begin{align*}
      \delta_{r,j}(\omega) = \sum_{c_k,c_{k+1}, \dots,c_{b_r} \in \{0,1\}} & [K_{r}(\omega_{1},\dots,\omega_{j-1},\omega_{j}, c_{j+1},\dots, c_{b_r})-K_{r}(\omega_{1},\dots,\omega_{j-1},c_{j}, c_{j+1},\dots, c_{b_r})] \\
                                         & \times\mathbb{P}(\omega_{j}= c_{j}, \omega_{j+1}= c_{j+1}, \dots, \omega_{b_r}=c_{b_r})
    \end{align*}

    The difference in $K_{r}$ induced by flipping the state of a single edge
    $e_{j}$ (from 0 to 1) is at most 1 (it either connects two existing clusters
    or does not). Thus, $|\delta_{r,j}|$ is uniformly bounded by $1$. We require
    a lower bound on $\Var_{p}(K_{r})$. We denote
    $\mathcal{C}_{B_r}(\omega;x)$ simply as $\mathcal{C}_{r}(\omega;x)$ and
    use the identity
    \[
      K_{r}= \sum_{x \in B_r}|\mathcal{C}_{r}(x)|^{-1}.
    \]
    Since $\omega \mapsto |\mathcal{C}_{r}(x;\omega)|^{-1}$ is non-increasing
    with respect to the configuration ordering, the FKG inequality yields:
    \[
      \Var_{p}(K_{r}) \geq \sum_{x \in B_r}\Var_{p}( |\mathcal{C}_{r}(x)|^{-1}
      ).
    \]
    Since $1 \leq |\mathcal{C}_{r}(x)| \leq |B_{r}|$, we have
    \[
      \E_{p}[|\mathcal{C}_{r}(x)|^{-1}] = \Prob_{p}(|\mathcal{C}_{r}(x)|=1)
      + \sum_{k=2}^{|B_r|}\frac{1}{k}\Prob_{p}(|\mathcal{C}_{r}(x)|=k),
    \]
    then
    \[
      \E_{p}[|\mathcal{C}_{r}(x)|^{-1}] \leq \Prob_{p}(|\mathcal{C}_{r}(x)|
      = 1) + \frac{1}{2}\Prob_{p}(|\mathcal{C}_{r}(x)| \ge 2) = 1 - \frac{1}{2}
      \Prob_{p}(|\mathcal{C}_{r}(x)| \ge 2).
    \]
    Thus,
    $1 - \E_{p}[|\mathcal{C}_{r}(x)|^{-1}] \ge \frac{1}{2}\Prob_{p}(|\mathcal{C}
    _{r}(x)| \ge 2)$. Since the random variable $|\mathcal{C}_{r}(x)|^{-1}$ only
    takes values in $\{1, \frac{1}{2}, \frac{1}{3}, \dots\}$, the variance
    is bounded below by
    \begin{align*}
      \Var_{p}( |\mathcal{C}_{r}(x)|^{-1}) & \ge (1 - \E_{p}[|\mathcal{C}_{r}(x)|^{-1}])^{2}\Prob_{p}(|\mathcal{C}_{r}(x)| = 1) \\
                         & \ge \frac{1}{4}\Prob_{p}^{2}[|\mathcal{C}_{1}(x)| \ge 2] \cdot (1-p)^{|N_G(x)|},
    \end{align*}
    where the last inequality holds since
    $\mathcal{C}_{1}(x) \subseteq \mathcal{C}_{r}(x)$ (for $r \ge 1$) and
    $N_{G[B_r]}(x) \leq N_{G}(x)$. By quasi-transitivity, $V$ partitions
    into finitely many orbits. Within each orbit $\mathcal{O}_{i}$,
    $\Prob_{p}[|\mathcal{C}_{1}(x)| \ge 2] = c_{i}$ and $N_{G}(x) = d_{i}$
    are constant. Since $G$ is locally finite, connected, and $p \in (0,1)$,
    we have $c_{i}> 0$ and $d_{i}< \infty$ for all $i$. Let $c_{\min}:= \min_{i}
    c_{i}> 0$ and $d_{\max}:= \max_{i}d_{i}< \infty$. Then, for all $x \in V$,
    \[
      \Var_{p}( |\mathcal{C}_{r}(x)|^{-1}) \ge \frac{1}{4}c_{\min}^{2}(1-p
      )^{d_{\max}}=: \tau(p) > 0.
    \]
    Then, there exists $\tau(p) > 0$ such that $\Var_{p}(K_{r}) \geq \tau(p)|
    B_{r}|$. Since $|\delta_{r,j}| \leq 1$, we have
    \[
      |X_{r,j}| \leq 1 / \sqrt{\tau(p) |B_{r}|}.
    \]
    This implies
    \[
      \max_{j}|X_{r,j}| \leq 1 / \sqrt{\tau(p) |B_{r}|},
    \]
    which converges to 0 as $r \to \infty$ (since $G$ is infinite),
    satisfying (2). It also implies
    \[
      \E_{p}[\max_{j}|X_{r,j}|^{2}] \leq 1 / (\tau(p) |B_{r}|),
    \]
    which satisfies (1).

    For condition (3), we must show
    \[
      \sum_{j=1}^{b_r}X^{2}_{r,j}= \frac{\sum_{j=1}^{b_r}\delta^{2}_{r,j}}{\Var_{p}(K_{r})}
      \to 1 \quad \text{in probability.}
    \]
    Recalling the following identity for martingales
    $\Var_{p}(K_{r}) = \sum_{j=1}^{b_r}\E_{p}[\delta_{r,j}^{2}]$, then we have
    to show
    \[
      \frac{\sum_{j=1}^{b_r}(\delta^{2}_{r,j}- \E_{p}[\delta^{2}_{r,j}])}{\Var_{p}(K_{r})}
      \to 0 \quad \text{in probability.}
    \]
    Since $G$ is locally finite and quasi-transitive,
    $b_{r}\leq d_{\max}|B_{r}|$. Combined with $\Var_{p}(K_{r}) \ge \tau (p )
    |B_{r}|$, we have
    \[
      \Var_{p}(K_{r}) \ge (\tau(p)/d_{\max}) b_{r}=: c b_{r}
    \]
    for some $c>0$. It thus suffices to show that for any $\varepsilon > 0$,
    \[
      \lim_{r\to\infty}\Prob_{p}\left( \left| \sum_{j=1}^{b_r}(\delta^{2}_{r,j}
      - \E_{p}[\delta^{2}_{r,j}]) \right| > \varepsilon c b_{r}\right) = 0.
    \]

    Let $e \in E$. We denote its endpoints by $x_{1}(e)$ and $x_{2}(e)$, ordered
    such that $d(x_{1}(e), o) \leq d(x_{2}(e), o)$. For a configuration
    $\omega \in \Omega$, $m \ge 0$, and $j \in \{1,2\}$, we define $\mathcal{C}
    '_{m}(\omega; x_{j}(e))$ as the connected component of $x_{j}(e)$ within
    the subgraph induced by $B_{m}(x_{1}(e))$ using only the open edges in $E
    (\omega) \setminus \{e\}$. We write $B_{m}(e)$ for $B_{m}(x_{1}(e))$. Formally,
    $y \in \mathcal{C}'_{m}( \omega;x_{j}(e))$ if and only if
    $y \in B_{m}(e)$ and there exists a path $(v_{0}, \dots, v_{n})$ such
    that:
    \begin{itemize}
      \item $v_{0}= x_{j}(e)$ and $v_{n}= y$;

      \item $\{v_{i}\}_{i=0}^{n}\subset B_{m}(e)$;

      \item $\{v_{i}, v_{i+1}\} \in E(\omega) \setminus \{e\}$ for all $i \in
        \{0, \dots, n-1\}$.
    \end{itemize}

    Let $e \in E_{r}:= E(G[B_{r}])$. We define $D_{r}(e,m)$ as the event that
    the endpoints of $e$ are disconnected in $B_{m}(e)$ considering only edges
    in $E_{r}(\omega )\setminus \{e\}$, and both of their respective
    clusters reach the boundary of $B_{m}(e)$.

    More precisely, $D_{r}(e,m)$ occurs if and only if both of the following
    conditions hold:
    \begin{enumerate}
      \item[(i)] $\mathcal{C}'_{m}(x_{1}(e)) \cap \mathcal{C}'_{m}(x_{2}(e)
        ) = \varnothing$;

      \item[(ii)] $\exists u \in \mathcal{C}'_{m}(x_{1}(e))$ and $\exists v
        \in \mathcal{C}'_{m}(x_{2}(e))$ such that $d(u,x_{1}(e)) = d(v,x_{2}
        (e)) = m$.
    \end{enumerate}

    This event is illustrated in Figure \ref{example_event_D}.

    \begin{figure}[h!]
      \centering
      \begin{tikzpicture}[>=Stealth, scale=0.8]
        \draw[gray!15, thin] (0,0) grid (8,8);
        \foreach \x in {0, ..., 8} \foreach \y in {0, ..., 8}
        \fill[gray!50] (\x,\y) circle (1.5pt);

        \coordinate (x1) at (4,4);
        \coordinate (x2) at (5,4);

        \draw[line width=1pt, black, dashed]
          (4,8) --
          (8,4) --
          (4,0) --
          (0,4) --
          cycle;

        \node[font=\tiny, fill=white, inner sep=0.5pt]
          at
          (5.5, 0.5)
          {$\partial B_{4}(x_{1}(e))$};

        \draw[line width=2pt, black] (x1) -- (x2);

        \begin{scope}[
          color=red!80!black,
          line width=1.2pt,
          line cap=round,
          line join=round
        ]
          \draw (x1) -- (3,4) -- (3,5) -- (3,3) -- (1,3);
          \draw (3,4) -- (2,4) -- (2,6);
          \draw (2,3) -- (2,2) -- (3,2);
        \end{scope}

        \begin{scope}[
          color=blue!80!black,
          line width=1.2pt,
          line cap=round,
          line join=round
        ]
          \draw (x2) -- (5,7);
          \draw (5,6) --
            (6,6) --
            (6,3) --
            (6,2) --
            (5,2) --
            (5,3) -- (6,3);
          \draw (5,5) -- (4,5);
        \end{scope}

        \foreach \p in {x1, x2}
        { \fill[white] (\p) circle (2.2pt); \fill[black] (\p) circle (1.8pt); }

        \node[font=\scriptsize] at (3.75, 3.75) {$x_{1}(e)$};
        \node[font=\scriptsize] at (5.25, 3.75) {$x_{2}(e)$};
        \node[font=\small, above=2pt] at ($(x1)!0.5!(x2)$) {$e$};

        \node[text=red!80!black, font=\scriptsize]
          at
          (0.9,6)
          {$\mathcal{C}'_{4}(x_{1}(e))$};
        \node[text=blue!80!black, font=\scriptsize]
          at
          (7.1,6)
          {$\mathcal{C}'_{4}(x_{2}(e))$};
      \end{tikzpicture}
      \caption{Example of $D_{r}(e,4)$ event in $\mathbb{Z}^{2}$.}
      \label{example_event_D}
    \end{figure}
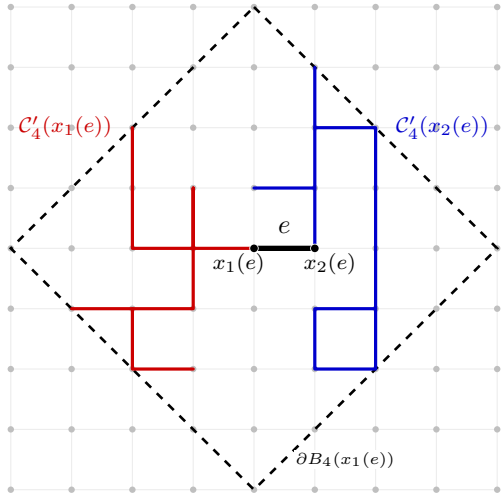

    Now, we need the following lemma.

    \begin{lemma}
      \label{Growthlemma} Let $f$ be the growth function of a quasi-transitive
      graph. The function $f$ has subexponential growth if and only if
      there exists a function $g:\mathbb{N}\to\mathbb{N}$ satisfying the
      following conditions:
      \begin{enumerate}
        \item \label{itm:g_infty} $\lim\limits_{r\to\infty}g(r)=\infty$;

        \item \label{itm:g_bound} $0<g(r)<r$ for all sufficiently large $r$;

        \item \label{itm:g_limit} $\displaystyle \lim\limits_{r\to\infty}
          \frac{f(r)-f(r-g(r))}{f(r)}=0$.

        \item \label{itm:f(2g)} For any fixed $k \geq 0$,
          $\lim\limits_{r\to\infty}\frac{f(2g(r)+k)}{f(r)}= 0,$
      \end{enumerate}
    \end{lemma}

    \begin{proof}
      First, we prove the existence of such a function $g$ assuming $f$
      has subexponential growth. The growth function $f$ of an infinite quasi-transitive
      graph with subexponential growth has the following properties:
      \begin{enumerate}
        \item \label{prop:monotonicity} {Monotonicity:} The function $f$
          is non-decreasing and $f(r) \to \infty$ as $r \to \infty$,
          since $G$ is infinite and connected.

        \item \label{prop:folner} If a quasi-transitive graph has subexponential
          growth then the sequence of its balls $(B_{r})$ is a F\o
          lner sequence. (\cite{garrido2013introduction}, Theorem 3.8).

          This implies that the boundary of the balls is
          asymptotically small relative to their volume. As a direct consequence,
          the relative growth rate
          \begin{equation*}
            \nabla(r):= \frac{f(r)-f(r-1)}{f(r-1)}. \label{def:delta}
          \end{equation*}
          satisfies
          \[
            \lim_{r\to\infty}\nabla(r) = \lim_{r\to\infty}\frac{f(r)-f(r-1)}{f(r-1)}
            = 0.
          \]
      \end{enumerate}

      \textit{ Construction of $g(r)$.} We first define a tail function
      $M(R)$ that measures the supremum of the growth rate from an index
      $R$ onward:
      \begin{equation*}
        \label{def:M}M(R):= \sup_{r\ge R}\nabla(r),\qquad R\in\mathbb{R}.
      \end{equation*}
      By definition, $M(R)$ is non-increasing, and by \eqref{prop:folner},
      $M(R)\downarrow 0$ as $R\to\infty$.

      Now, for each $r\ge 2$, we define $g(r)$ adaptively, based on the
      convergence speed of $M(R)$:
      \begin{equation*}
        g(r):= \min\!\left\{\left\lfloor \frac{r}{2}\right\rfloor,\left\lfloor
        \frac{1}{4}\log_{c}(r)\right\rfloor,\ \left\lfloor\, M\!\Big(\left
        \lfloor\frac{r}{2}\right\rfloor\Big)^{-\frac{1}{2}}\right\rfloor\right
        \},\label{def:g}
      \end{equation*}
      where $c = \max\{2,|N(x)| : x\in V\}$.

      From the definition, $0<g(r)\leq \lfloor r/2\rfloor<r$, yielding \eqref{itm:g_bound}.
      Since $\lim_{r\to\infty}\lfloor r/2 \rfloor = \infty$, we have that
      $\lim_{r\to\infty}M(\lfloor r/2 \rfloor) = 0$, which implies
      \[
        \lim_{r\to\infty}\lfloor M(\lfloor r/2 \rfloor)^{-1/2}\rfloor = \infty
        .
      \]
      Also, we have that
      $\lim_{r\to\infty}\left\lfloor \frac{1}{4}\log_{c}(r)\right\rfloor =
      \infty$. Since $g(r)$ is the minimum of three sequences that diverge
      to infinity, it also diverges, proving \eqref{itm:g_infty}.

      \textit{Verification of the Convergence Condition.} To prove \eqref{itm:g_limit},
      we express the difference as a sum of increments:
      \[
        f(r)-f(r-g(r))=\sum_{k=0}^{g(r)-1}\big(f(r-k)-f(r-k-1)\big).
      \]
      Dividing by $f(r)$, Property~\ref{prop:monotonicity} implies $f(r) \ge
      f(r-k-1)$, and thus
      \[
        \frac{f(r)-f(r-g(r))}{f(r)}\leq \sum_{k=0}^{g(r)-1}\frac{f(r-k)-f(r-k-1)}{f(r-k-1)}
        = \sum_{k=0}^{g(r)-1}\nabla(r-k).
      \]
      For each term $\nabla(r-k)$ in the sum, the index $m = r-k$ satisfies
      $r-g(r)+1 \leq m \leq r$. Since $g(r)\leq \lfloor r/2\rfloor$, the smallest
      index satisfies $m \ge r-(g(r)-1) \ge r-(\lfloor r/2\rfloor-1) \ge \lceil
      r/2 \rceil$. Thus, $m \ge \lfloor r/2 \rfloor$ for all $k$ in the sum.
      By the definition of $M(R)$ in \eqref{def:M}:
      \[
        \nabla(r-k)\ \leq \ M\!\Big(\Big\lfloor\frac{r}{2}\Big\rfloor\Big
        ) \quad\text{for all }k=0,\dots,g(r)-1.
      \]
      Substituting this uniform bound into the sum, which has $g(r)$ terms:
      \[
        \frac{f(r)-f(r-g(r))}{f(r)}\leq g(r)\cdot M\!\Big(\Big\lfloor\frac{r}{2}
        \Big\rfloor\Big).
      \]
      Then, we use the definition of $g(r)$ from \eqref{def:g}. By the definition
      of $g(r)$ we have
      \[
        g(r)\cdot M\!\Big(\Big\lfloor\frac{r}{2}\Big\rfloor\Big) \leq \left
        \lfloor M\!\Big(\Big\lfloor\frac{r}{2}\Big\rfloor\Big)^{-\frac{1}{2}}
        \right\rfloor \cdot M\!\Big(\Big\lfloor\frac{r}{2}\Big\rfloor\Big
        ).
      \]
      For $x>0$, the inequality
      $\lfloor x^{-1/2}\rfloor \cdot x \leq x^{-1/2}\cdot x = x^{1/2}$
      holds. With $x=M(\lfloor r/2\rfloor)$, we obtain:
      \[
        \frac{f(r)-f(r-g(r))}{f(r)}\ \leq \ M\!\Big(\Big\lfloor\frac{r}{2}
        \Big\rfloor\Big)^{1/2}.
      \]
      Since $M(R)\to 0$ as $R\to\infty$, the right-hand side tends to 0 as
      $r\to\infty$.

      Finally, considering the properties of the graph $G$, the volume of
      the ball must satisfy $c^{r}\geq f(r) \ge r$. Therefore,
      \[
        0 \leq \lim_{r\to\infty}\frac{f(2g(r) + k)}{f(r)}\leq \lim_{r\to\infty}
        \frac{f(3g(r))}{f(r)}\leq \lim_{r\to\infty}\frac{c^{\frac{3}{4}\log_c(r)}}{r}
        = \lim_{r\to\infty}\frac{r^{3/4}}{r}=0.
      \]

      Conversely, suppose that $f$ has exponential growth. To obtain a contradiction,
      assume there exists a function $g$ satisfying the desired conditions.
      By the definition of exponential growth, there exist
      $R_{0}\in \mathbb{N}$ and a constant $a > 1$ such that, for all $r \geq
      R_{0}$:
      \begin{equation*}
        f(r) > a^{r}.
      \end{equation*}

      By condition \ref{itm:g_limit}, we have:
      \[
        \lim_{r \to \infty}\frac{f(r - g(r))}{f(r)}= 1.
      \]

      Since $\lim_{r \to \infty}g(r) = \infty$, for sufficiently large $r$,
      we have $g(r) \geq 1$. Due to the monotonicity of $f$, it follows
      that $f(r - g(r)) \leq f(r - 1) \leq f(r)$. Dividing by $f(r)$, we
      obtain:
      \[
        \frac{f(r - g(r))}{f(r)}\leq \frac{f(r - 1)}{f(r)}\leq 1.
      \]

      So as $r \to \infty$, we conclude that:
      \begin{equation*}
        \lim_{r \to \infty}\frac{f(r - 1)}{f(r)}= 1 \implies \lim_{r \to
        \infty}\frac{f(r)}{f(r - 1)}= 1.
      \end{equation*}

      Let $\varepsilon = \frac{a - 1}{2}$. Since $a > 1$, it follows that $\varepsilon
      > 0$ and $1 + \varepsilon < a$. By the previous limit, there exists $R
      _{1}\in \mathbb{N}$ (with $R_{1}> R_{0}$) such that, for all
      $r \geq R_{1}$:
      \[
        \frac{f(r)}{f(r - 1)}< 1 + \varepsilon.
      \]

      For any $r > R_{1}$, we can express $f(r)$ as a telescoping product:
      \[
        f(r) = f(R_{1}) \prod_{s=R_1+1}^{r}\frac{f(s)}{f(s - 1)}.
      \]

      Applying the upper bound $(1 + \varepsilon)$ to each term of the product,
      we get:
      \[
        f(r) < f(R_{1}) (1 + \varepsilon)^{r - R_1}= \frac{f(R_{1})}{(1 +
        \varepsilon)^{R_1}}(1 + \varepsilon)^{r}.
      \]

      Combining this with the lower bound $f(r) > a^{r}$, we have:
      \[
        a^{r}< f(r) < C \cdot (1 + \varepsilon)^{r},
      \]
      where $C = \frac{f(R_{1})}{(1 + \varepsilon)^{R_1}}$ is a positive
      constant. Dividing the inequality by $a^{r}$ yields:
      \[
        1 < \frac{f(r)}{a^{r}}< C \left( \frac{1 + \varepsilon}{a}\right)
        ^{r}.
      \]

      Since $1 + \varepsilon < a$, the ratio $\frac{1 + \varepsilon}{a}$
      is strictly less than $1$. Therefore:
      \[
        \lim_{r \to \infty}C \left( \frac{1 + \varepsilon}{a}\right)^{r}=
        0.
      \]

      Taking $r \to \infty$, we obtain the contradiction $1 \leq 0$, establishing
      that no such function $g$ exists for graphs of exponential growth.
    \end{proof}

    Henceforth, we set $m = g(r)$ and write $D_{r}(e)$ for $D_{r}(e,g(r))$. The
    complementary event $D_{r}^{c}(e)$ occurs if and only if one of the following
    holds:
    \begin{itemize}
      \item there is a path between the endpoints of $e$ inside $B_{g(r)}(e
        )$ without passing through $e$;

      \item the cluster of one of its endpoints after the removal of $e$ remains
        completely in the interior of $B_{g(r)}(e)$.
    \end{itemize}
    We denote the first event by $J_{r}(e)$ and the second by $U_{r}(e)$, where
    \[
      J_{r}(e) = [\mathcal{C}'_{g(r)}(x_{1}(e)) \cap \mathcal{C}'_{g(r)}(x_{2}
      (e)) \ne \varnothing]
    \]
    and
    \[
      U_{r}(e) = [\mathcal{C}'_{g(r)}(x_{1}(e)) \cap \mathcal{C}'_{g(r)}(x_{2}
      (e)) = \varnothing] \cap [\exists i\in \{1,2\}\, \text{ s.t. }\,\mathcal{C}
      '_{g(r)}(x_{i}(e)) \subseteq B_{g(r)-1}(e)].
    \]

    Examples of these events are illustrated in Figures
    \ref{example_event_J} and \ref{example_event_H}, respectively.

    \begin{figure}[h!]
      \centering
      \begin{minipage}{0.48\textwidth}
        \centering
        \begin{tikzpicture}[>=Stealth, scale=0.45]
          \draw[gray!15, thin] (0,0) grid (8,8);
          \foreach \x in {0, ..., 8} \foreach \y in {0, ..., 8}
          \fill[gray!50] (\x,\y) circle (1.2pt);

          \coordinate (x1) at (4,4);
          \coordinate (x2) at (5,4);

          \draw[line width=0.8pt, black, dashed]
            (4,8) --
            (8,4) --
            (4,0) --
            (0,4) --
            cycle;
          \node[font=\tiny, fill=white, inner sep=0.5pt]
            at
            (5.5, 0.5)
            {$\partial B_{4}(x_{1}(e))$};

          \draw[line width=1.5pt, black] (x1) -- (x2);

          \begin{scope}[
            line width=1pt,
            line cap=round,
            line join=round
          ]
            \draw (x1) -- (3,4) -- (3,5) -- (3,3) -- (1,3);
            \draw (3,4) -- (2,4) -- (2,6);
            \draw (2,3) -- (2,2) -- (3,2);
          \end{scope}

          \begin{scope}[
            line width=1pt,
            line cap=round,
            line join=round
          ]
            \draw (x2) -- (5,7);
            \draw (5,6) --
              (6,6) --
              (6,3) --
              (6,2) --
              (5,2) --
              (5,3) -- (6,3);
            \draw (5,5) -- (3,5);
          \end{scope}

          \foreach \p in {x1, x2} { \fill[white] (\p) circle (1.8pt); \fill[black] (\p) circle (1.4pt); }

          \node[font=\tiny, above=1pt] at ($(x1)!0.5!(x2)$) {$e$};
          \node[font=\tiny] at (3.6,6) {$\mathcal{C}_{4}(x_{1}(e))$};
        \end{tikzpicture}
        \caption{Example of $J_{r}(e,4)$ event.}
        \label{example_event_J}
      \end{minipage}
      \hfill
      \begin{minipage}{0.48\textwidth}
        \centering
        \begin{tikzpicture}[>=Stealth, scale=0.45]
          \draw[gray!15, thin] (0,0) grid (8,8);
          \foreach \x in {0, ..., 8} \foreach \y in {0, ..., 8}
          \fill[gray!50] (\x,\y) circle (1.2pt);

          \coordinate (x1) at (4,4);
          \coordinate (x2) at (5,4);

          \draw[line width=0.8pt, black, dashed]
            (4,8) --
            (8,4) --
            (4,0) --
            (0,4) --
            cycle;
          \node[font=\tiny, fill=white, inner sep=0.5pt]
            at
            (5.5, 0.5)
            {$\partial B_{4}(x_{1}(e))$};
          \draw[line width=1.5pt, black] (x1) -- (x2);

          \begin{scope}[
            color=red!80!black,
            line width=1pt,
            line cap=round,
            line join=round
          ]
            \draw (x1) -- (3,4) -- (3,5) -- (3,3) -- (2,3);
            \draw (3,4) -- (1,4);
          \end{scope}

          \begin{scope}[
            color=blue!80!black,
            line width=1pt,
            line cap=round,
            line join=round
          ]
            \draw (x2) -- (5,7);
            \draw (5,6) --
              (6,6) --
              (6,3) --
              (6,2) --
              (5,2) --
              (5,3) -- (6,3);
            \draw (5,5) -- (4,5);
          \end{scope}

          \foreach \p in {x1, x2}
          { \fill[white] (\p) circle (1.8pt); \fill[black] (\p) circle (1.4pt); }

          \node[font=\tiny] at (3.7, 3.7) {$x_{1}(e)$};
          \node[font=\tiny] at (5.3, 3.7) {$x_{2}(e)$};
          \node[font=\tiny, above=1pt] at ($(x1)!0.5!(x2)$) {$e$};
          \node[text=red!80!black, font=\tiny]
            at
            (0.9,6.3)
            {$\mathcal{C}'_{4}(x_{1}(e))$};
          \node[text=blue!80!black, font=\tiny]
            at
            (7.1,6.3)
            {$\mathcal{C}'_{4}(x_{2}(e))$};
        \end{tikzpicture}
        \caption{Example of $U_{r}(e,4)$ event.}
        \label{example_event_H}
      \end{minipage}
    \end{figure}

    Then, $J_{r}(e)$ and $U_{r}(e)$ depend only on the state of the edges in
    $B_{g(r)}(e)$ and $D^{c}_{r}(e)$ is their disjoint union.

    Given a configuration $c =(c_{1}, \dots, c_{b_r})\in\{0,1\}^{b_r}$ of $G[
    B_{r}]$ we denote the variation in the number of clusters $K_{r}$ due to
    flipping the state of edge $e_{j}$ by
    \[
      \Delta_{r,j}(c) := K_{r}(c_{1},\dots,c_{j-1},c_{j},c_{j+1}, \dots,c_{b_r}
      )-K_{r}(c_{1},\dots,c_{j-1},\tilde{c}_{j},c_{j+1},\dots,c_{b_r}),
    \]
    where $\tilde{c}_{j}= 1 -c_{j}$. Note that $\delta_{r,j}$ is
    $\mathcal{F}_{j}$-measurable, i.e., it depends only on the status of the
    first $j$ edges. Given $c_{1}, \dots, c_{j}\in \{0,1\}$, we have
    \begin{equation}
      \label{eq:delta_def}\delta_{r,j}(c_{1},\dots,c_{j}) = \sum_{c_{j+1},
      \dots,c_{b_r}\in \{0,1\}}\Delta_{r,j}(c)\mathbb{P}_{p}(\tilde{c}_{j},
      c_{j+1}, \dots, c_{b_r}).
    \end{equation}
    We also define
    \begin{equation}
      \label{eq:deltap_def}\delta'_{r,j}(c_{1},\dots,c_{j}) = \sum_{c_{j+1},
      \dots,c_{b_r}\in \{0,1\}}\mathbbm{1}_{D^c_r(e_j)}(c_{1},\dots,c_{j-1}
      , c_{j+1}, \dots, c_{b_r})\Delta_{r,j}(c)\mathbb{P}_{p}(\tilde{c}_{j}
      , c_{j+1}, \dots, c_{b_r}),
    \end{equation}
    if $e_{j}\in B_{r-g(r)}$, and $\delta'_{r,j}=0$ otherwise.

    By the triangular inequality we have
    \begin{align}
      \frac{1}{\tau(p)b_{r}}\Biggl|\sum_{j=1}^{b_r}(\delta^{2}_{r,j}- \E_{p}[\delta^{2}_{r,j}])\Biggr| & \leq \frac{1}{\tau(p)b_{r}}\Biggl|\sum_{j=1}^{b_r}[(\delta'_{r,j})^{2}- \E_{p}[\delta'_{r,j}]^{2}]\Biggr|\label{eq:delta-1}   \\
                                                       & + \frac{1}{\tau(p)b_{r}}\Biggl|\sum_{j=1}^{b_r}[\E_{p}[\delta'_{r,j}]^{2}- \E_{p}[\delta^{2}_{r,j}]]\Biggl|\label{eq:delta-2} \\
                                                       & + \frac{1}{\tau(p)b_{r}}\Biggl|\sum_{j=1}^{b_r}(\delta^{2}_{r,j}- (\delta'_{r,j})^{2})\Biggl|\label{eq:delta-3}.
    \end{align}

    For the first term, by Chebyshev's inequality,
    \begin{align*}
      \mathbb{P}_{p}\Biggl(\frac{1}{\tau(p)b_{r}} & \Biggl|\sum_{j=1}^{b_r}[(\delta'_{r,j})^{2}- \E_{p}[\delta'_{r,j}]^{2}]\Biggl|> \varepsilon/3 \Biggr) \leq \frac{9}{(\varepsilon \tau(p)b_{r})^{2}}\Var\Biggl(\sum_{j=1}^{b_r}(\delta'_{r,j})^{2}\Biggl).
    \end{align*}
    Note that in the event $J_{r}(e_{j})$, the flipping of state of $e_{j}$
    does not change the number of clusters in $K_{r}$. In the event $U_{r}(e_{j}
    )$, it changes the number of clusters in $K_{r}$ by exactly one, as at least
    one of the clusters is not connected to the border of $B_{g(r)}(e)$.
    Moreover, if two edges $e_{j}, e_{k}$ satisfy $d(x_{1}(e_{j}),x_{1}(e_{k}
    ))> 2g(r)$, their associated variables are independent, and thus $\mathrm{Cov}
    \bigl((\delta'_{r,j})^{2}, (\delta'_{r,k})^{2}\bigr) = 0$. Using Hölder's
    inequality, we bound the variance term:
    \begin{align*}
      \Var\Biggl(\sum_{j=1}^{b_r}(\delta'_{r,j})^{2}\Biggl) & = \sum_{k=1}^{b_r}\sum_{j=1}^{b_r}\mathrm{Cov}\Bigl((\delta'_{r,j})^{2}, (\delta'_{r,k})^{2}\Bigl)                                \\
                                  & \leq \sum_{k=1}^{b_r}\sum_{j \,:\, d(x_1(e_j),x_1(e_k)) \leq 2g(r)}\mathrm{Cov}\Bigl((\delta'_{r,j})^{2}, (\delta'_{r,k})^{2}\Bigl)               \\
                                  & \leq \sum_{k=1}^{b_r}\sum_{j \,:\, d(x_1(e_j),x_1(e_k)) \leq 2g(r)}\E_{p}\Bigl[(\delta'_{r,j})^{2}(\delta'_{r,k})^{2}\Bigl]                   \\
                                  & \leq \sum_{k=1}^{b_r}\sum_{j \,:\, d(x_1(e_j),x_1(e_k)) \leq 2g(r)}\Bigl(\E_{p}\Bigl[(\delta'_{r,j})^{4}\Bigl]\E_{p}\Bigl[(\delta'_{r,k})^{4}\Bigl]\Bigl)^{1/2} \\
                                  & \leq \sum_{k=1}^{b_r}\sum_{j \,:\, d(x_1(e_j),x_1(e_k)) \leq 2g(r)}1                                              \\
                                  & \leq b_{r}\cdot \max_{x \in B_{r-g(r)}}b_{2g(r)}(x),
    \end{align*}
    where $b_{2g(r)}(x) = |\{e \in E: d(x_{1}(e),x) \leq 2g(r) \}|$.

    Since the graph is quasi-transitive, there exists $y \in V$ such that $b_{2g(r)}
    (y) = \sup_{x \in V}b_{2g(r)}(x)$. Let $\tilde{r}= d(y,o)$. Then for any
    $r > 2\tilde{r}$, we have
    \[
      b_{2g(r)+\tilde{r}}\geq b_{2g(r)}(y) \geq \sup_{x \in B_{r-g(r)}}b_{2g(r)}
      (x),
    \]
    so
    \begin{align*}
      \mathbb{P}_{p}\Biggl(\frac{1}{\tau(p)b_{r}}\Biggl|\sum_{j=1}^{b_r}[(\delta'_{r,j})^{2}- \E_{p}[\delta'_{r,j}]^{2}]\Biggl|> \varepsilon/3 \Biggl) & \leq \frac{9}{\varepsilon^{2}\tau^{2}(p)}\frac{b_{2g(r)+\tilde{r}}}{b_{r}}                        \\
                                                                               & \leq \frac{9c}{\varepsilon^{2}\tau^{2}(p)}\frac{|B_{2g(r)+\tilde{r}}|}{|B_{r}|}\underset{r \to \infty}{\longrightarrow}0,
    \end{align*}
    by Lemma \ref{Growthlemma}.

    For the second term \eqref{eq:delta-2} we note that
    \begin{align*}
      \frac{1}{\tau(p)b_{r}}\Biggl|\sum_{j=1}^{b_r}[\E_{p}[\delta'_{r,j}]^{2}- \E_{p}[\delta^{2}_{r,j}]]\Biggl| & \leq \frac{1}{\tau(p)b_{r}}\E_{p}\Biggl[\sum_{j=1}^{b_r}\bigl|\delta'_{r,j}- \delta_{r,j}\bigl|\cdot\bigl|\delta'_{r,j}+ \delta_{r,j}\bigl|\Biggl] \\
                                                            & \leq \frac{2}{\tau(p)b_{r}}\sum_{j=1}^{b_r}\E_{p}\bigl|\delta'_{r,j}- \delta_{r,j}\bigl|.
    \end{align*}
    We can separate this into two sums: one over the edges in $E_{r-g(r)}$, and
    the other over $E_{r}\setminus E_{r-g(r)}$, where
    $\delta'_{r,j}\equiv 0$. As $|\Delta_{r,j}(c)|\leq 1,$ for any $r, j$ and
    $c$, we have
    \[
      \E_{p}\bigl|\delta'_{r,j}- \delta_{r,j}\bigl| \leq \sum_{c \in \{0,1\}^{b_r}}
      \mathbbm{1}_{D_r(e_j)}(c)\mathbb{P}_{p}(c)=\mathbb{P}_{p}(D_{r}(e_{j}
      )).
    \]
    For the second sum (over $e_{j}\in E_{r}\setminus E_{r-g(r)}$), we have $\E
    _{p}\bigl|\delta'_{r,j}-\delta_{r,j}\bigl| = \E_{p}\bigl |\delta_{r,j}\bigl
    |\leq 1.$ Then there exist constants $a,b >0$ such that
    \[
      \frac{2}{\tau(p)b_{r}}\sum_{j=1}^{b_r}\E_{p}\bigl|\delta'_{r,j}- \delta
      _{r,j}\bigl| \leq a\sum_{i=1}^{n}\mathbb{P}_{p}(D_{r}(f_{i})) + b \frac{|B_{r}|-|B_{r-g(r)}|}{|B_{r}|}
      ,
    \]
    where for each $i$, $f_{i}$ is chosen as an edge in the $i$-th class
    under the $\Aut(G)$ action that minimizes the distance from the origin
    $o$. The right term of the sum goes to zero by Lemma \ref{Growthlemma}.

    It is well known that an infinite, connected, nonamenable graph must have
    exponential growth. Since $G$ is an amenable quasi-transitive graph, by \cite[Theorem
    7.6]{lyons2017probability}, there is at most one infinite cluster a.s. We
    show that this implies the first term $a\sum_{i=1}^{n}\mathbb{P}_{p}(D_{r}
    (f_{i}))$ also vanishes. The event $D_{r}(f_{i}) \cap [\omega(f_{i})=0]$
    implies the existence of two clusters of size at least $r$. Since this sequence
    of events is decreasing as $r\to\infty$, by the continuity of the
    probability measure and the independence between the events
    $D_{r}(f_{i})$ and $[\omega(f_{i})=0]$, we have:
    \begin{align*}
      \lim_{r\to\infty}\Prob_{p}(D_{r}(f_{i})) & = \frac{1}{\Prob_{p}(\omega(f_{i})=0)}\lim_{r\to\infty}\Prob_{p}(D_{r}(f_{i}) \cap [\omega(f_{i})=0]) \\
                           & =\frac{1}{1-p}\Prob_{p}\Biggl(\bigcap_{r=1}^{\infty}[D_{r}(f_{i})\cap [\omega(f_{i})=0]] \Biggr)    \\
                           & \leq \frac{1}{1-p}\Prob_{p}(\text{there exist at least two infinite clusters}) = 0.
    \end{align*}

    For the third term \eqref{eq:delta-3}, notice that
    \begin{align*}
      \mathbb{P}_{p}\Biggl(\frac{1}{\tau(p)b_{r}}\Biggl|\sum_{j=1}^{b_r}[\delta^{2}_{r,j}- (\delta'_{r,j})^{2}]\Biggl|> \varepsilon/3\Biggl) & \leq \frac{3}{\varepsilon}\E_{p}\Biggl[\frac{1}{\tau(p)b_{r}}\Biggl|\sum_{j=1}^{b_r}[\delta_{r,j}^{2}- (\delta'_{r,j})^{2}]\Biggl|\Biggl] \\
                                                                           & \leq \frac{3}{\varepsilon\tau(p)b_{r}}\E_{p}\Biggl[\sum_{j=1}^{b_r}|\delta^{2}_{r,j}- (\delta'_{r,j})^{2}|\Biggl],
    \end{align*}
    which is, up to constants, equal to an intermediate step of the second term.
  \end{proof}

  \section{ Functionals}
  \label{sec:functionals}

  Although the following theorem is originally stated in \cite{manfred259ergodic}
  in the context of topological groups and Haar measure, we present it here
  adapted to the setting of Cayley graphs. In this discrete and countable setting,
  the underlying group $\Gamma$ satisfies the original hypotheses, with its Haar
  measure simply coinciding with the counting measure.

  \begin{theorem}[\cite{manfred259ergodic}, Theorem 8.13]
    Let $G = \Cay(\Gamma, S)$ be an infinite, amenable Cayley graph, equipped
    with an independent Bernoulli bond percolation $(\Omega, \mathcal{F}, \Prob
    _{p})$. Then, for any F{\o}lner sequence $(A_{r})_{r\geq 1}$ in $V$ and any
    $L^{2}$ random variable $X:\Omega\to \mathbb{R}$, we have:
    \[
      \frac{1}{|A_{r}|}\sum_{x\in A_r}X\circ\varphi_{x^{-1}}\underset{L^2}{\longrightarrow}
      \E_{p}[X],
    \]
    as $r \to \infty$. \label{ergodic_amenable_theorem}
  \end{theorem}

  \begin{proof}[Proof of Theorem \ref{TLC_func}]
    Since $H$ is an LO subgroup, we fix a total ordering $\leq_{H}$ on $H$. We
    extend this to a total order on $\Gamma$ lexicographically: for
    $g, g' \in \Gamma$, let $g = hx_{i}$ and $g' = \bar{h}x_{j}$ be their unique
    decompositions with $h, \bar{h}\in H$. We set $g \leq g'$ if $i < j$, or
    if $i = j$ and $h \leq_{H}\bar{h}$. This vertex ordering induces a total
    order $\preccurlyeq$ on $E$. For each edge $e=\{u,v\} \in E$, let
    $\min(e)$ and $\max(e)$ denote the smaller and larger endpoints of $e$
    with respect to $\leq$, respectively. We define $e \preccurlyeq f$ if
    $(\min(e), \max(e)) \leq_{\text{lex}}(\min(f), \max(f))$. It follows from
    the left-invariance of $\leq_{H}$ that the orderings $\leq$ on $V$ and $\preccurlyeq$
    on $E$ are left-invariant under $H$ action.

    Let $E_{r}$ and $E_{r}(x)$ denote the sets of edges with endpoints in $A_{r}$
    and $xA_{r}$, respectively.

    Define the filtration $(\mathcal{F}_{e})_{e \in E}$ by setting
    $\mathcal{F}_{0}= \{\varnothing, \Omega\}$ and
    $\mathcal{F}_{e}= \sigma(\{ \omega_{e'}: e' \preccurlyeq e \})$. For a fixed $r
    > 0$, let $E_{r}= \{e_{1}, \dots, e_{b_r}\}$ be the set of edges with
    both endpoints in $A_{r}$, ordered according to the restriction of $\preccurlyeq.$
    To simplify notation, we write $D_{j}, \omega^{j}$, and
    $\mathcal{F}_{j}$ for $D_{e_j}, \omega^{e_j}$, and $\mathcal{F}_{e_j}$,
    respectively.

    \begin{figure}[htbp]
      \centering
      \begin{tikzpicture}[
        scale=0.6,
        every node/.style={font=\small},
        vtx/.style={circle,inner sep=0pt,minimum size=3.2pt,fill=black},
        revealed/.style={blue!65!black,line width=1.2pt,},
        current/.style={red,line width=2.0pt},
        unrevealed/.style={gray!55,line width=0.6pt},
        axislabel/.style={font=\footnotesize,gray!70}
      ]
        \def\xmin{-3} \def\xmax{4} \def\ymin{-3} \def\ymax{3} \def\a{1} \def\b{1}
        \pgfmathtruncatemacro{\xmaxmone}{\xmax-1} \pgfmathtruncatemacro{\ymaxmone}{\ymax-1}
        \pgfmathtruncatemacro{\amone}{\a-1} \pgfmathtruncatemacro{\bmone}{\b-1}
        \foreach \x in {\xmin,...,\xmaxmone} { \foreach \y in {\ymin,...,\ymax} { \draw[unrevealed] (\x,\y) -- (\x+1,\y); } }
        \foreach \x in {\xmin,...,\xmax} { \foreach \y in {\ymin,...,\ymaxmone} { \draw[unrevealed] (\x,\y) -- (\x,\y+1); } }
        \foreach \x in {\xmin,...,\amone} { \foreach \y in {\ymin,...,\ymax} { \draw[revealed] (\x,\y) -- (\x+1,\y); } }
        \foreach \x in {\xmin,...,\amone} { \foreach \y in {\ymin,...,\ymaxmone} { \draw[revealed] (\x,\y) -- (\x,\y+1); } }
        \foreach \y in {\ymin,...,\ymaxmone} { \draw[revealed] (\a,\y) -- (\a,\y+1); }
        \foreach \y in {\ymin,...,\b} { \draw[revealed] (\a,\y) -- (\a+1,\y); }
        \foreach \y in {\ymin,...,\bmone} { \draw[revealed] (\a+1,\y) -- (\a+1,\y+1); }
        \foreach \x in {\xmin,...,\xmax} { \foreach \y in {\ymin,...,\ymax} { \node[vtx] at (\x,\y) {}; } }
        \coordinate (L) at (\a,\b);
        \coordinate (R) at (\a+1,\b);
        \coordinate (M) at ($(L)!0.5!(R)$);
        \draw[current] (L) -- (R);
        \node[red, below=3pt] at (M) {$e_{j}$};
        \def\stabr{2.35}
        \draw[dashed]
          ($(M)+(\stabr,0)$) --
          ($(M)+(0,\stabr)$) --
          ($(M)+(-\stabr,0)$) --
          ($(M)+(0,-\stabr)$) --
          cycle;
        \node[above=9pt]
          at
          ($(M)+(0,\stabr)$)
          {\footnotesize stabilization};
        \begin{scope}[shift={(5.95,2.15)}]
          \draw[black!40] (0,-0.8) rectangle (7.6,1.55);
          \draw[revealed] (0.15,0.9) -- (0.55,0.9);
          \node[anchor=west]
            at
            (0.7,0.9)
            {Preceding edges ($e \prec e_{j}$)};
          \draw[unrevealed] (0.15,0.2) -- (0.55,0.2);
          \node[anchor=west]
            at
            (0.7,0.2)
            {Succeeding edges ($e \succ e_{j}$)};
          \draw[current] (0.15,-0.4) -- (0.55,-0.4);
          \node[anchor=west] at (0.7,-0.4) {Edge $e_{j}$};
        \end{scope}
        \begin{scope}[
          draw=gray!55,
          line width=0.6pt,
          dash pattern=on 2pt off 2pt,
          line cap=round
        ]
          \def\stub{1} \foreach \y in {\ymin,...,\ymax} { \draw (\xmax,\y) -- ++(\stub,0); }
          \foreach \x in {\xmin,...,\xmax} { \draw (\x,\ymax) -- ++(0,\stub); }
          \foreach \x in {\xmin,...,\xmax} { \draw (\x,\ymin) -- ++(0,-\stub); }
          \foreach \y in {\ymin,...,\ymax} { \draw (\xmin,\y) -- ++(-\stub,0); }
        \end{scope}
        \node[axislabel] at (\xmax+0.9,\ymax+0.45) {$\cdots$};
        \node[axislabel] at (\xmin-0.9,\ymax+0.45) {$\cdots$};
        \node[axislabel] at (\xmax+0.9,\ymin-0.45) {$\cdots$};
        \node[axislabel] at (\xmin-0.9,\ymin-0.45) {$\cdots$};
        \node at (\xmin-0.3, \ymax+0.95) {$\mathbb{Z}^{2}$};
      \end{tikzpicture}
      \caption{Lexicographic martingale filtration on $\mathbb{Z}^{2}$: the
      preceding edges (blue) represent the history $\mathcal{F}_{j-1}$,
      the red edge is the current edge $e_{j}$, and the gray edges are the
      succeeding. The dashed diamond is the $d_{G}$-ball (graph-distance ball)
      around $e_{j}$, depicting the stabilization neighborhood.}
      \label{fig:filtration_stabilization}
    \end{figure}
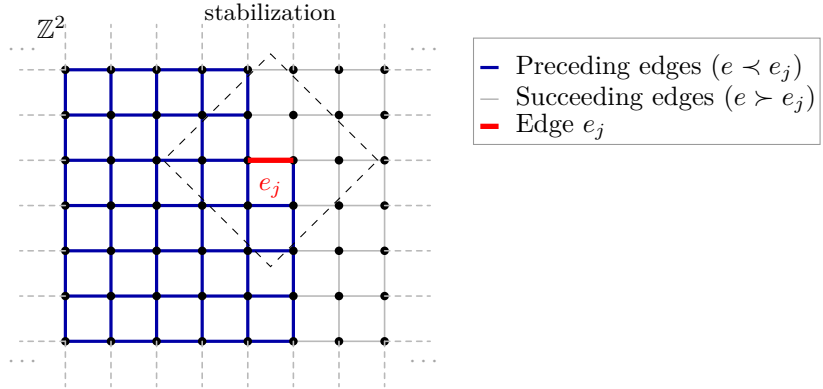

    Note that for each $r > 0$, since $F(A_{r})$ depends only on the edges
    within $A_{r}$, it is independent of any edges outside $A_{r}$ that lie between
    $e_{j-1}$ and $e_{j}$ in the global ordering $\preccurlyeq$. So, the sequence
    \[
      \delta_{r,j}:= \E_{p}[ F(A_{r}) \mid \mathcal{F}_{j}] - \E_{p}[ F(A_{r}
      ) \mid \mathcal{F}_{j-1}], \quad j = 1, \dots, b_{r},
    \]
    forms a martingale difference sequence with respect to the filtration $\{
    \mathcal{F}_{j}\}$. By the telescoping property, we have that $F(A_{r}) -
    \E_{p}[F(A_{r})] = \sum_{j=1}^{b_r}\delta_{r,j}$. Hence, by the orthogonality
    of martingale differences, we obtain
    \[
      \operatorname{Var}_{p}(F(A_{r})) = \sum_{j=1}^{b_r}\E_{p}[\delta_{r,j}
      ^{2}].
    \]

    By applying the martingale central limit theorem \cite{mcleish1974dependent},
    it suffices to verify the following conditions as $r \to \infty$:
    \begin{enumerate}
      \item $\sup_{r\geq 1}\E_{p}\left[ \max_{1\leq j \leq b_r}|A_{r}|^{-1}
        \delta_{r,j}^{2}\right] < \infty,$

      \item $\max_{1\leq j \leq b_r}|A_{r}|^{-1/2}|\delta_{r,j}| \xrightarrow
        {\Prob_p}0,$

      \item $|A_{r}|^{-1}\sum_{j=1}^{b_r}\delta_{r,j}^{2}\xrightarrow{\Prob_p}
        \sigma^{2}.$
    \end{enumerate}

    Let $S$ be the finite, symmetric generating set such that $G = \Cay (\Gamma
    , S)$.

    To construct an $H$-fundamental set of edges, fix
    $X = \{x_{1}, \dots, x_{n}\}$ a set of right coset representatives of
    $H$ in $\Gamma$. Define the map $\phi: H \times X \times S \to E$ by $\phi
    (h, x_{i}, s) = \{hx_{i}, hx_{i}s\}$, where $E$ denotes the edge set of
    $\Cay(\Gamma, S)$. This map is well-defined and surjective. For any
    $e = \{g, gs\} \in E$, the unicity of the representation $g = hx_{i}$ yields
    $e = \phi(h, x_{i}, s)$.

    Each edge $e \in E$ has exactly two distinct pre-images. Indeed, if $e =
    \{hx_{i}, hx_{i}s\}$, it can also be represented by the vertex $hx_{i}s$.
    Letting $\bar{h}x_{j}$ be the unique representation of $hx_{i}s$ in $Hx_{j}$,
    for some $x_{j}$, we have $e = \{ \bar{h}x_{j}, \bar{h}x_{j}s^{-1}\} = \phi
    (\bar{h}, x_{j}, s^{-1})$. These pre-images $p_{1}= (h, x_{i}, s)$ and
    $p_{2}= (\bar{h}, x_{j}, s^{-1})$ are distinct, otherwise $s = s^{-1}$ and
    $hx_{i}s = hx_{i}$, implying $s = 1 \notin S$.

    Furthermore, any pre-image $(\tilde{h}, x_{k}, t)$ must satisfy either $(
    \tilde{h}, x_{k}) = (h, x_{i})$ and $t = s$, or
    $(\tilde{h}, x_{k}) = (\bar{h}, x_{j})$ and $t = s^{-1}$. By the uniqueness
    of the coset representation, we have that $\phi$ is 2-to-1.

    Thus, the set
    \[
      \tilde{E}= \phi(1, X, S) = \{ \{x_{i}, x_{i}s\} : x_{i}\in X, s \in S
      \}
    \]
    forms an $H$-fundamental set of edges on $E$. More specifically, under
    the action of $H$, each unoriented edge has exactly two representations in
    the parametrization
    \[
      (h,x_{i},s)\mapsto \{hx_{i},hx_{i}s\}
    \], corresponding to the two orientations.

    For each $1 \leq j \leq b_{r}$, we fix an element $h_{j}\in H$ such that
    $e_{j}= \{h_{j}^{-1}x_{i}, h_{j}^{-1}x_{i}s\}$ for some $x_{i}\in X, s \in
    S$, and let $\varphi_{j}$ denote the translation $\varphi_{h_j}$. Since $\tilde
    {E}$ is finite, $\tilde{E}\subseteq E_{r}$ for all sufficiently large
    $r$. Under the restriction of $\preccurlyeq$ to $E_{r}$, we denote by
    $\varphi(j)$ the index of the fundamental edge associated with $e_{j}$.

    For notational simplicity, we suppress the $r$-dependence of $h_{j}, \varphi
    _{j}$, and $\varphi(j)$.

    Let $\omega \in \Omega$. The martingale difference can be expressed as:
    \begin{align*}
      \delta_{r,j}(\omega) & = \sum_{c_j, c_{j+1}, \dots, c_{b_r} \in \{0,1\}}[F((\omega_{1}, \dots, \omega_{j},c_{j+1}, \dots, c_{b_r}),A_{r}) - F((\omega_{1}, \dots, c_{j},c_{j+1}, \dots, c_{b_r}),A_{r})] \\
                 & \quad \times \Prob_{p}(c_{j}, c_{j+1}, \dots, c_{b_r})                                                              \\
                 & = \sum_{c_j, c_{j+1}, \dots, c_{b_r} \in \{0,1\}}D_{j}((\omega_{1}, \dots, \omega_{j},c_{j+1}, \dots, c_{b_r}),A_{r})\Prob_{p}(c_{j}, c_{j+1}, \dots, c_{b_r})          \\
                 & = \E_{p}[D_{j}(A_{r})\mid\mathcal{F}_{j}](\omega).
    \end{align*}

    For the first condition, it follows from conditional Jensen's inequality,
    the tower property, and the fact that every vertex has degree $|S|$, that
    \begin{align*}
      |A_{r}|^{-1}\E_{p}\biggl[\max_{1\leq j\leq b_r}\delta^{2}_{r,j}\biggr] & \leq \frac{|S|}{b_{r}}\sum_{j=1}^{b_r}\E_{p}[\delta^{2}_{r,j}]              \\
                                           & \leq \frac{|S|}{b_{r}}\sum_{j=1}^{b_r}\E_{p}[\E^{2}_{p}[D_{j}(A_{r})\mid\mathcal{F}_{j}]] \\
                                           & \leq \frac{|S|}{b_{r}}\sum_{j=1}^{b_r}\E_{p}[D^{2}_{j}(A_{r})].
    \end{align*}

    Now, fix $e_{j}\in E_{r}$. By the stationarity of the functional $F$, we
    have
    \begin{align*}
      D^2_{j}(\omega,A_{r}) & = [F(\omega,A_{r}) - F(\omega^{j},A_{r})]^2                                                  \\
                & = [F(\tilde{\varphi}_{j}(\omega),\varphi_{j}(A_{r})) - F((\tilde{\varphi}_{j}(\omega))^{\varphi(j)},\varphi_{j}(A_{r}))]^2 \quad \text{a.s.} \\
                & = D^2_{\varphi(j)}(\tilde{\varphi}_{j}(\omega),h_{j}A_{r}).
    \end{align*}
  Since $\tilde{\varphi}_{j}:\Omega\to\Omega$ is measure-preserving under
$\Prob_p$, we have
\[
\E_{p}[D^2_{j}(A_{r})] =
\E_{p}[D^2_{\varphi(j)}(\varphi_{j}(A_{r}))].
\]for
    $e_{\varphi(j)}\in \tilde{E}$. Then,
    \begin{align*}
      |A_{r}|^{-1}\E_{p}\biggl[\max_{1\leq j\leq b_r}\delta^{2}_{r,j}\biggr] & \leq |S|\max_{e \in \tilde{E}}\sup_{A\in \mathcal{A}}\E_{p}[D^{2}_{e}(A)] < \infty,
    \end{align*}
    by the finite moment assumption.

    Now, we verify condition 2. Let $\varepsilon > 0$. By Markov's
    inequality, we have
    \begin{align*}
      \Prob_{p}\biggl(\max_{1\leq j \leq b_r}|A_{r}|^{-1/2}|\delta_{r,j}| \geq \varepsilon\biggr) & \leq \sum_{j=1}^{b_r}\Prob_{p}\Bigl(|\delta_{r,j}|\geq\varepsilon |A_{r}|^{1/2}\Bigl)                       \\
                                                    & \leq \frac{1}{\varepsilon^{\gamma}|A_{r}|^{\gamma/2}}\sum_{j=1}^{b_r}\E_{p}\Bigl[|\delta_{r,j}|^{\gamma}\Bigl]          \\
                                                    & \leq \frac{1}{\varepsilon^{\gamma}|A_{r}|^{\gamma/2}}\sum_{j=1}^{b_r}\E_{p}\Bigl[|D_{j}(A_{r})|^{\gamma}\Bigl]          \\
                                                    & \leq \frac{1}{\varepsilon^{\gamma}|A_{r}|^{-1+\gamma/2}}\max_{e \in \tilde{E}}\sup_{A\in \mathcal{A}}\E_{p}[|D_{e}(A)|^{\gamma}],
    \end{align*}
    which vanishes as $r \to \infty$ for some $\gamma > 2$ given the finite moment
    condition.

    It remains to verify the third condition. Recall that for each
    $e_{j}\in E$, the sequence of random variables $(D_{j}(A_{r}))_{r\in\mathbb{N}}$
    is a.s. equal to a sequence of the form $D_{\varphi(j)}(h_{j}A_{r})$, where
    $e_{\varphi(j)}\in \tilde{E}$. By the sequential stability condition, this
    sequence converges in probability to a limiting random variable $D_{\varphi(j)}
    ^{\infty}$. Hence, there exists a random variable $D_{j}^{\infty}$ such that
    $D_{j}(A_{r}) \xrightarrow{\Prob_p}D_{j}^{\infty}$ as $r \to \infty$.

    For $x \in V, r \geq 1$ and $e_{j}\in E_{r}(x)$, let
    \[
      \delta^{x}_{r,j}(\omega)= \E_{p}[D_{j}(xA_{r})\mid\mathcal{F}_{j}](\omega
      )
    \]
    and
    \[
      \delta_{j}(\omega)= \E_{p}[D_{j}^{\infty}\mid\mathcal{F}_{j}](\omega)
      .
    \]

    Let $\mathfrak{F}_{j}= \{e \in E :e \prec e_{j}\}$. For any
    $\rho \in \{0,1\}^{E_r(x)\setminus \mathfrak{F}_{j-1}}$, let $\omega\mathbbm
    {1}_{E_r(x)\cap\mathfrak{F}_j}+\rho$ denote the configuration that coincides
    with $\omega$ on $E_{r}(x) \cap \mathfrak{F}_{j}$ and with $\rho$ on the
    remaining edges of $E_{r}(x)$. Furthermore, let $\omega\mathbbm{1}_{E_r(x)\cap\mathfrak{F}_{j-1}}
    +\rho$ denote the configuration that coincides with $\omega$ on $E_{r}(x)
    \cap \mathfrak{F}_{j-1}$ and with $\rho$ on the remaining edges of $E_{r}
    (x)$. Note that the only difference between
    $\omega\mathbbm{1}_{E_r(x)\cap\mathfrak{F}_j}+\rho$ and
    $\omega\mathbbm{1}_{E_r(x)\cap\mathfrak{F}_{j-1}}+\rho$ is that the
    former takes the value $\omega_{j}$ at $e_{j}$, while the latter takes the
    value $\rho_{j}$ at $e_{j}$. Then, the conditional expectation admits following
    the representation:
    \begin{align*}
      \E_{p}[D_{j}(xA_{r})\mid\mathcal{F}_{j}](\omega) & = \sum_{\rho \in \{0,1\}^{E_r(x)\setminus \mathfrak{F}_{j-1}}}[F(\omega\mathbbm{1}_{E_r(x)\cap\mathfrak{F}_j}+\rho,xA_{r})-F(\omega\mathbbm{1}_{E_r(x)\cap\mathfrak{F}_{j-1}}+\rho,xA_{r})]\Prob_{p}(\rho).
    \end{align*}
    Since $\varphi_{j}$ is a translation by an element of $H$, we have $e \prec
    e_{j}\iff \varphi_{j}(e) \prec \varphi_{j}(e_{j}) = e_{\varphi(j)}$.
    It follows that
    \begin{align*}
      \varphi_{j}(\mathfrak{F}_{j}) & = \{\varphi_{j}(e) \in E: e \prec e_{j}\}             \\
                      & = \{e\in E: \varphi^{-1}_{j}(e) \prec e_{j}\}           \\
                      & = \{e\in E: e \prec \varphi_{j}(e_{j})\}              \\
                      & = \{e\in E: e \prec e_{\varphi(j)}\} = \mathfrak{F}_{\varphi(j)}.
    \end{align*}

    Thus, for any $\omega \in \Omega$, the transformation $\tilde{\varphi}_{j}$
    yields:
   Let \(\tilde{\varphi}_j:=\tilde{\varphi}_{h_j}\). Since \(\varphi_j\) preserves the
order on edges, we have
\[
e \preccurlyeq e_j \iff \varphi_j(e) \preccurlyeq e_{\varphi(j)}.
\]
Hence
\[
\tilde{\varphi}_j^{-1}(\mathcal F_j)=\mathcal F_{\varphi(j)}.
\]

Moreover, by stationarity of \(F\),
\[
D_j(A_r)\circ \tilde{\varphi}_j^{-1}
=
D_{\varphi(j)}(h_jA_r).
\]
Since \(\tilde{\varphi}_j\) is measure-preserving under \(\Prob_p\), the conditional
expectation transforms according to
\[
\bigl(\E_p[D_j(A_r)\mid \mathcal F_j]\bigr)\circ \tilde{\varphi}^{-1}_{j}
=
\E_p[D_{\varphi(j)}(h_jA_r)\mid \mathcal F_{\varphi(j)}].
\]
Therefore,
\[
\delta_{r,j}\circ \tilde{\varphi}^{-1}_{j}
=
\E_p[D_{\varphi(j)}(h_jA_r)\mid \mathcal F_{\varphi(j)}]
=
\delta^{h_j}_{r,\varphi(j)}.
\]

    To obtain an analogous result for $\delta_{j}$, recall that
    $D_{j}(\omega,A_{r}) = D_{\varphi(j)}(\tilde{\varphi}_{j}(\omega),h_{j}A_{r}
    )$, where $e_{\varphi(j)}\in \tilde{E}$. The finite moment condition
    implies the existence of $\gamma > 2$ and $C > 0$ such that
    \[
      \sup_{r \geq 1}\E_{p}[|D_{j}(xA_{r})|^{\gamma}] = \sup_{r \geq 1}\E_{p}
      [|D_{\varphi(j)}((h_{j}x)A_{r})|^{\gamma}] < C,
    \]
    yielding uniform integrability of the sequence
    $(D_{j}(xA_{r}))_{r \geq 1}$. Since $D_{j}(xA_{r}) \xrightarrow{\Prob_p}D
    _{j}^{\infty}$, the sequence converges in $L^{1}$.

    Hence, by the conditional Jensen inequality and the tower property, we
    have
    \begin{align*}
      \E_{p}[|\delta^{x}_{r,j}-\delta_{j}|] & = \E_{p}[|\E_{p}[D_{j}(xA_{r}) - D_{j}^{\infty}\mid \mathcal{F}_{j}]|] \\
                          & \leq \E_{p}[|D_{j}(xA_{r}) - D_{j}^{\infty}|] \to 0
    \end{align*}
    as $r \to \infty$, which shows that $\delta_{r,j}^{x}\xrightarrow{\Prob_p}
    \delta_{j}$. Finally, since
    \[
      \delta_{r,j}\circ \tilde{\varphi}_{j}^{-1}\xrightarrow{\Prob_p}\delta
      _{j}\circ \tilde{\varphi}_{j}^{-1}
    \]
    and
    \[
      \delta_{r,\varphi(j)}^{h_j}\xrightarrow{\Prob_p}\delta_{\varphi(j)}.
    \]
    The identity
    \[
      \delta_{r,j}\circ \tilde{\varphi}_{j}^{-1}= \delta_{r,\varphi(j)}^{h_j}
      \quad\text{a.s.}
    \]
    implies that
    \[
      \delta_{j}\circ \tilde{\varphi}_{j}^{-1}= \delta_{\varphi(j)}\quad \text{a.s.}
    \]

    Next, we show that $(\delta_{r,j}^{x})^{2}$ converges to $\delta_{j}^{2}$
    in $L^{1}(\Prob_{p})$. By the Cauchy-Schwarz inequality,
    \begin{align*}
      \E_{p}[|(\delta_{r,j}^{x})^{2}-\delta^{2}_{j}|] & = \E_{p}[|\delta_{r,j}^{x}-\delta_{j}| \cdot |\delta_{r,j}^{x}+\delta_{j}|]  \\
                              & \leq \|\delta_{r,j}^{x}-\delta_{j}\|_{2}\|\delta_{r,j}^{x}+\delta_{j}\|_{2}.
    \end{align*}
    Observe that, by Jensen's inequality and the finite moment condition,
    \begin{align*}
      \E_{p}[|\delta_{r,j}^{x}+\delta_{j}|^{2}] & =\E_{p}\bigl[ \E_{p}^{2}[D_{j}(xA_{r})+D_{j}^{\infty}\mid\mathcal{F}_{j}]\bigr] \\
                            & \leq \E_{p}[(D_{j}(xA_{r})+D_{j}^{\infty})^{2}]                 \\
                            & = \E_{p}[(D_{\varphi(j)}((h_{j}x)A_{r})+D_{\varphi(j)}^{\infty})^{2}] < \infty.
    \end{align*}
    It remains to show that $\|\delta_{r,j}^{x}- \delta_{j}\|_{2}\to 0$ as $r
    \to \infty$. By the contractivity of the conditional expectation in
    $L^{2}$, we have
    \begin{align*}
      \|\delta_{r,j}^{x}-\delta_{j}\|_{2} & =\|\E_{p}[D_{j}(xA_{r})-D_{j}^{\infty}\mid \mathcal{F}_{j}]\|_{2} \\
                        & \leq \|D_{j}(xA_{r})-D_{j}^{\infty}\|_{2}.
    \end{align*}
    
    By the finite moment assumption, there exists $\gamma > 2$ such that
    \[
      \sup_{r\geq 1}\mathbb{E}_{p}[|D_{j}(xA_{r})|^{\gamma}] < \infty.
    \]
    Denoting $\phi(a) = a^{\gamma/2}$, we have
    \[
      \sup_{r\geq 1}\mathbb{E}_{p}[\phi(|D_{j}(xA_{r})|^{2})] < \infty.
    \]
    Thus, by \cite[Theorem 4.6.2]{durrett2019probability}, the sequence $|D_{j}
    (xA_{r})|^{2}$ is uniformly integrable.

    Since $D_{j}(xA_{r}) \xrightarrow{\Prob_p}D_{j}^{\infty}$, by
    \cite[Theorem 2.3.4]{durrett2019probability}, we have $|D_{j}(xA_{r}) - D
    _{j}^{\infty}|^{2}\xrightarrow{\Prob_p}0$. Next, we show that the sequence
    $(|D_{j}(xA_{r}) - D_{j}^{\infty}|^{2})_{r \geq 1}$ is uniformly integrable.
    Furthermore, the inequality 
\[
|D_j(xA_r)-D_j^\infty|^2 \le 2|D_j(xA_r)|^2 + 2|D_j^\infty|^2
\]
shows that, as the family $\bigl(|D_j(xA_r)|^2\bigr)_{r\ge1}$ is uniformly integrable and $|D_j^\infty|^2 \in L^1(\Prob_p)$, the family of squared differences 
\[
\bigl(|D_j(xA_r)-D_j^\infty|^2\bigr)_{r\ge1}
\]
is also uniformly integrable. In view of the fact that 
\[
|D_j(xA_r)-D_j^\infty|^2 \xrightarrow[r\to\infty]{\Prob_p} 0,
\]
by \cite[Theorem 4.6.3]{durrett2019probability} we conclude that $\E_p\left[|D_j(xA_r)-D_j^\infty|^2\right] \to 0$, which is to say that 
\[
D_j(xA_r) \xrightarrow[r \to \infty]{L^2(\Prob_p)} D_j^\infty.
\]
Hence,
    \[
      \frac{1}{|A_{r}|}\sum_{j=1}^{b_r}((\delta_{r,j}^{x})^{2}-\delta_{j}^{2}
      ) \xrightarrow{L^1}0.
    \]

    For $x_{i}\in X$ and $s$, let
    \[
      E_{i,s}:= \phi(H,x_{i},s) = \{\{hx_{i},hx_{i}s\}: h \in H\}
    \]
    denote the orbit of the edge $\{x_{i}, x_{i}s\}$ by $H$. Given a subset $A
    \subseteq V$, we define the restriction of $E_{i,s}$ to $A$ as
    \[
      E_{i,s}(A) = \{\{hx_{i},hx_{i}s\} \in E_{i,s}: hx_{i}\in A, hx_{i}s \in
      A\}.
    \]
    Also $H_{i,s}(A)$ be the set of group elements whose associated edges
    are contained in $A$, namely,
    \[
      H_{i,s}(A) = \{h \in H : \phi(h,x_{i},s) \subseteq A\}.
    \]

    Thus, we observe that the fundamental set of edges decomposes a F{\o}lner
    sequence into a finite family of F{\o}lner sequences.
    \begin{lemma}
      For any $(x_{i},s)\in X\times S$, the sequence
      $(H_{i,s}(A_{r}))_{r\ge 1}$ is a F{\o}lner sequence in $H$.
    \end{lemma}

    \begin{proof}
      Let
      \[
        Y:=S^{-1}\ \cup\ \{x_{i}^{-1}x_{j}:\ 1\le i,j\le n\}.
      \]
      The set $Y$ encapsulates both the local geometry of the graph and the
      algebraic transitions between the right cosets of $H$, thereby
      allowing for a comparison between the sizes of the intersections of
      the sequence with different cosets.

      Since $\Gamma$ is amenable and $Y$ is finite, we may fix a F{\o}lner
      sequence $(A_{r})$ in $\Gamma$ such that
      \begin{equation}
        \label{eq:right-inv-Y}\frac{|A_{r}y\triangle A_{r}|}{|A_{r}|}\longrightarrow
        0\qquad\text{for all }y\in Y.
      \end{equation}
      Indeed, for each $m\in\mathbb{N}$ there exists a finite set
      $A\subset\Gamma$ with $|gA\triangle A|<\frac{1}{m}|A|$ and
      $|Ag\triangle A|<\frac{1}{m}|A|$ for all $g\in Y$; choosing $A=A_{m}$
      gives \eqref{eq:right-inv-Y}.

      For each $i$ set
      \[
        K_{i,r}:=\{h\in H:\ hx_{i}\in A_{r}\}, \qquad\text{so that}\qquad
        |K_{i,r}|=|A_{r}\cap Hx_{i}| \ \ \text{and}\ \ |A_{r}|=\sum_{i=1}
        ^{n}|K_{i,r}|.
      \]
      For $y_{ij}:=x_{i}^{-1}x_{j}$, right-multiplication by $y_{ij}$ is a
      bijection $Hx_{i}\to Hx_{j}$. Hence
      \[
        |K_{i,r}|=|A_{r}\cap Hx_{i}|=\bigl|(A_{r}y_{ij})\cap Hx_{j}\bigr
        |.
      \]
      It follows from,
      \[
        \bigl||K_{i,r}|-|K_{j,r}|\bigr| \le |A_{r}y_{ij}\triangle A_{r}|
        =o(|A_{r}|),
      \]
      that \eqref{eq:right-inv-Y}. Since $\sum_{i=1}^{n}|K_{i,r}|=|A_{r}|$,
      it follows that
      \begin{equation}
        \label{eq:Ki-asymp}\frac{|K_{i,r}|}{|A_{r}|}\longrightarrow \frac{1}{n}
        \qquad (1\le i\le n).
      \end{equation}

      Now recall
      \[
        H_{i,s}(A_{r})=\{h\in H:\ hx_{i}\in A_{r},\ hx_{i}s\in A_{r}\}.
      \]
      If $h\in K_{i,r}\setminus H_{i,s}(A_{r})$, then $hx_{i}\in A_{r}$
      and $hx_{i}s\notin A_{r}$, i.e.
      $hx_{i}\in A_{r}\setminus A_{r}s^{-1}\subset A_{r}s^{-1}\triangle A_{r}$.
      Using the bijection $h\mapsto hx_{i}$ from $H$ onto $Hx_{i}$, we get
      \[
        |K_{i,r}\setminus H_{i,s}(A_{r})| =\bigl|\{hx_{i}:\ h\in K_{i,r}\setminus
        H_{i,s}(A_{r})\}\bigr| \le |A_{r}s^{-1}\triangle A_{r}| =o(|A_{r}
        |)
      \]
      by \eqref{eq:right-inv-Y}. Together with \eqref{eq:Ki-asymp} this yields
      \begin{equation}
        \label{eq:His-size}\frac{|H_{i,s}(A_{r})|}{|A_{r}|}\longrightarrow
        \frac{1}{n}.
      \end{equation}

      Finally, fix $k\in H$. If
      $h\in kH_{i,s}(A_{r})\triangle H_{i,s}(A_{r})$, then by unwinding the
      definition of $H_{i,s}(\cdot)$ one has
      \[
        hx_{i}\in kA_{r}\triangle A_{r}\quad\text{or}\quad hx_{i}s\in kA_{r}
        \triangle A_{r}.
      \]
      Using again the bijection $h\mapsto hx_{i}$,
      \[
        |kH_{i,s}(A_{r})\triangle H_{i,s}(A_{r})| \le |kA_{r}\triangle A_{r}
        |+|(kA_{r}\triangle A_{r})s^{-1}| =2|kA_{r}\triangle A_{r}|.
      \]
      Divide by $|H_{i,s}(A_{r})|$ and use \eqref{eq:His-size} and the (left)
      F{\o}lner property of $(A_{r})$ (applied to $k$) to obtain
      \[
        \frac{|kH_{i,s}(A_{r})\triangle H_{i,s}(A_{r})|}{|H_{i,s}(A_{r})|}
        \longrightarrow 0.
      \]
      Hence $(H_{i,s}(A_{r}))$ is F{\o}lner in $H$.
    \end{proof}
    Note that for an edge $e_{j}$, we previously denoted its associated fundamental
    edge by $e_{\varphi(j)}$. However, the fundamental edge associated with a
    given edge is determined not by the index $j$, but by the two classes
    $E_{i,s}$ to which $e_{j}$ belongs. Accordingly, for $e_{j}\in E_{i,s}$,
    we shall henceforth denote the index of its associated fundamental edge simply
    as $(i,s)$.

    Applying Theorem \ref{ergodic_amenable_theorem}, we obtain
    \begin{equation*}
      \frac{1}{|E_{i,s}(A_{r})|}\sum_{j : e_j \in E_{i,s}(A_r)}\delta_{j}^{2}
      = \frac{1}{|H_{i,s}(A_{r})|}\sum_{h_j \in H_{i,s}(A_r)}\delta_{(i,s)}
      ^{2}\circ \tilde{\varphi}_{h_j}^{-1}\xrightarrow{L^1}\E_{p}[\delta_{(i,s)}
      ^{2}].
    \end{equation*}

    The collection $E_{r}$ can be decomposed into the union of sets $E_{i,s}(
    A_{r})$ for $(x_{i},s) \in X \times S$, where each edge is indexed
    exactly twice. It follows that
    \begin{align*}
      \frac{1}{|A_{r}|}\sum_{j=1}^{b_r}\delta_{j}^{2} & = \frac{1}{2|A_{r}|}\sum_{(x_i,s) \in X \times S}\sum_{j : e_j \in E_{i,s}(A_r)}\delta^{2}_{j}                                                  \\
                              & = \frac{1}{2}\sum_{x_i \in X}\sum_{s \in S}\frac{|H_{i,s}(A_{r})|}{|A_{r}|}\left( \frac{1}{|H_{i,s}(A_{r})|}\sum_{h_j \in H_{i,s}(A_r)}\delta^{2}_{(i,s)}\circ \tilde{\varphi}^{-1}_{h_j}\right).
    \end{align*}

    Let
    \[
      a_{r}= \frac{|H_{i,s}(A_{r})|}{|A_{r}|}
    \]
    and
    \[
      Y_{r}= \frac{1}{|H_{i,s}(A_{r})|}\sum_{h_j\in H_{i,s}(A_r)}\delta^{2}
      _{(i,s)}\circ\tilde{\varphi}^{-1}_{h_j}.
    \]

    Setting $a = \frac{1}{n}$ and $Y = \E_{p}[\delta_{(i,s)}^{2}]$, we have
    $a_{r}\to a$ and, by Theorem~\ref{ergodic_amenable_theorem},
    $Y_{r}\xrightarrow{L^1(\Prob_p)}Y$. Moreover, by invariance,
    \[
      \E_{p}[|Y_{r}|]=\E_{p}[Y_{r}]=\E_{p}[\delta_{(i,s)}^{2}]<\infty \qquad
      \text{for all }r.
    \]
    Hence
    \begin{equation*}
      \E_{p}[|a_{r}Y_{r}- a Y|] \leq |a_{r}- a| \E_{p}[|Y_{r}|] + a \E_{p}[
      |Y_{r}- Y|] \to 0,
    \end{equation*}
    as $r \to \infty$. This establishes the $L^{1}$ convergence
    \begin{equation*}
      \frac{|H_{i,s}(A_{r})|}{|A_{r}|}\frac{1}{|H_{i,s}(A_{r})|}\sum_{h_j
      \in H_{i,s}(A_r)}\delta^{2}_{(i,s)}\circ \tilde{\varphi}^{-1}_{h_j}\xrightarrow
      {L^1}\frac{1}{n}\E_{p}[\delta_{(i,s)}^{2}],
    \end{equation*}
    which implies
    \begin{equation*}
      \frac{1}{|A_{r}|}\sum_{j=1}^{b_r}\delta_{j}^{2}\xrightarrow{L^1}\frac{1}{2n}
      \sum_{e \in \tilde{E}}\E_{p}[\delta_{e}^{2}].
    \end{equation*}
    Therefore, for any $x \in V$, we conclude that
    \begin{equation*}
      \frac{1}{|A_{r}|}\sum_{j=1}^{b_r}(\delta_{r,j}^{x})^{2}\xrightarrow{L^1}
      \frac{1}{2n}\sum_{e \in \tilde{E}}\E_{p}[\delta_{e}^{2}].
    \end{equation*}
    Observe that
    \begin{equation*}
      \frac{1}{|A_{r}|}\sum_{j=1}^{b_r}\delta_{r,j}^{2}\xrightarrow{\Prob_p}
      \frac{1}{2n}\sum_{e \in \tilde{E}}\E_{p}\left[ \E_{p}^{2}[D_{e}^{\infty}
      \mid \mathcal{F}_{e}] \right] = \sigma^{2}.
    \end{equation*}
    This completes the proof.
  \end{proof}

  We can also formulate an alternative result. By a standard double-counting
  argument on the $|S|$-regular graph $G$, we have $2 b_{r}= |S||A_{r}| - |\partial
  _{E}A_{r}|$, so that
  \begin{equation*}
    \frac{b_{r}}{|A_{r}|}= \frac{1}{2}\left( |S| - \frac{|\partial_{E}A_{r}|}{|A_{r}|}
    \right).
  \end{equation*}
  Since $(A_{r})$ is a F{\o}lner sequence, this boundary term vanishes asymptotically.
  Thus, the variance of
  \[
    \frac{F(A_{r})-\mathbb{E}_{p}[F(A_{r})]}{b_{r}^{1/2}}
  \]
  converges to
  \begin{align*}
    \lim_{r\to\infty}\frac{1}{b_{r}}\sum_{j=1}^{b_r}\delta_{r,j}^{2} & =\lim_{r\to\infty}\frac{|A_{r}|}{b_{r}}\frac{1}{|A_{r}|}\sum_{j=1}^{b_r}\delta_{r,j}^{2}=\frac{1}{n|S|}\sum_{e \in \tilde{E}}\E_{p}\left[ \E_{p}^{2}[D_{e}^{\infty}\mid \mathcal{F}_{e}] \right] \\
                                     & = \frac{1}{|\tilde{E}|}\sum_{e \in \tilde{E}}\E_{p}\left[ \E_{p}^{2}[D_{e}^{\infty}\mid \mathcal{F}_{e}] \right],
  \end{align*}
  noting that the final equality holds due to the bijection between $\tilde{E}$
  and $X \times S$.

  Since any Cayley graph is vertex-transitive, the assumption of
  subexponential growth implies that the sequence of metric balls
  $\{B_{n}\}_{n \in \mathbb{N}}$ constitutes a F{\o}lner sequence (see, e.g.,
  \cite[Theorem 3.8]{garrido2013introduction}). This leads to the following
  corollary.

  \begin{coro}
    \label{TLC_func_B_r} Assume $\Gamma$ has subexponential growth. Since
    the sequence of metric balls $(B_{r})$ is an exhaustive F{\o}lner
    sequence in any vertex-transitive graph, the conclusions of Theorem \ref{TLC_func}
    apply to $(B_{r})$. Specifically, under the same stabilization and moment
    conditions, we have
    \begin{enumerate}
      \item $\lim_{r\to\infty}|B_{r}|^{-1}\Var_{p}(F(B_{r})) = \sigma^{2}$,
        and

      \item $|B_{r}|^{-1/2}(F(B_{r}) - \E_{p}[F(B_{r})]) \overset{d}{\Longrightarrow}
        \mathcal{N}(0, \sigma^{2})$,
    \end{enumerate}
    where $\sigma^{2}$ is the asymptotic variance defined in Theorem
    \ref{TLC_func}.
  \end{coro}

  \begin{proof}[Proof of Theorem \ref{TLC_funcionais_polinomial}]
    We shall demonstrate that infinite, finitely generated groups of
    polynomial growth contain a left-orderable (LO) subgroup.

    \begin{theorem}[\cite{baumslag1971lecture}, Theorem 2.1]
      Let $\Gamma$ be a finitely generated nilpotent group. Then $\Gamma$
      is isomorphic to a finite-index subgroup of the direct product
      $D = A \times B$, where $A$ is finite and $B$ is torsion-free.
      \label{grupo_nilpotente_fatorado}
    \end{theorem}

    \begin{coro}
      Every finitely generated nilpotent group $\Gamma$ possesses a torsion-free,
      nilpotent subgroup of finite index. \label{grupo_nilpotente_torção}
    \end{coro}

    \begin{proof}
      By Theorem \ref{grupo_nilpotente_fatorado}, $\Gamma$ is, up to isomorphism,
      a finite-index subgroup of $D=A\times B$, where $A$ is finite and $B$
      is torsion-free. Let $H = \Gamma \cap B$. Note that $H$ is the intersection
      of two subgroups of $D$, hence it is a subgroup of $D$. Since $H \subseteq
      \Gamma$, it is a subgroup of $\Gamma$. Being a subgroup of $B$, $H$
      is torsion-free. Furthermore, since $\Gamma$ is nilpotent, $H$ is also
      nilpotent.

      We observe that $[D:B] = |A|$. Thus, the index of $H$ in $\Gamma$ satisfies:
      \[
        [\Gamma:H] = [D\cap\Gamma:B\cap\Gamma] \leq [D:B] = |A| < \infty.
      \]
    \end{proof}

    Recall that a central series of $\Gamma$ is a sequence of normal subgroups
    from $\Gamma$ down to the trivial group $\{1\}$, where each factor is
    the quotient of two consecutive subgroups. Our construction relies on the
    following structural property of these factors.

    \begin{theorem}[\cite{kargapolov1979fundamentals}, Theorem 17.2.2]
      Every finitely generated, torsion-free nilpotent group possesses a
      central series with infinite cyclic factors, i.e., factors isomorphic
      to $\mathbb{Z}$. \label{teorema de Malcev}
    \end{theorem}

    Given a set $A$, a (finite) coordinate system on $A$ is a collection of functions
    $f_{i}: A \to \mathbb{Z}$, with $i \in \{1,\dots, n\}$, such that the map
    $a \mapsto (f_{1}(a), \dots, f_{n}(a))$ is an injection from $A$ into $\mathbb{Z}
    ^{n}$.

    Let $\Gamma$ be a finitely generated, torsion-free nilpotent group and
    \[
      \{1\} = H_{n}\vartriangleleft H_{n-1}\vartriangleleft\dots\vartriangleleft
      H_{0}= \Gamma
    \]
    its central series with infinite cyclic factors. For each $i \in \{1, \dots
    , n\}$, the factor $H_{i-1}/H_{i}\simeq \mathbb{Z}$. Thus, there exists
    an element $a_{i}\in H_{i-1}$ such that its class $[a_{i}] = a_{i}H_{i}$
    is a generator of infinite order for the quotient.

    Hence, for any element $g\in H_{i-1}$, there exists a unique integer $k$
    such that $gH_{i}= a_{i}^{k}H_{i}$, implying $ga_{i}^{-k}\in H_{i}$.
    Thus, $H_{i-1}= \langle a_{i}, H_{i}\rangle$.

    By induction, we have:
    \begin{align*}
      H_{n-1}    & = \langle a_{n}, H_{n}\rangle = \langle a_{n}, \{1\}\rangle = \langle a_{n}\rangle, \\
      H_{n-2}    & = \langle a_{n-1}, H_{n-1}\rangle = \langle a_{n-1}, a_{n}\rangle,          \\
               & \vdots                                        \\
      \Gamma = H_{0} & =\langle a_{1}, \dots, a_{n}\rangle.
    \end{align*}

    Thus, each $g \in \Gamma$ can be uniquely expressed as
    \[
      g = a_{1}^{f_1(g)}a_{2}^{f_2(g)}\dots a_{n}^{f_n(g)}, \quad f_{i}(g)
      \in \mathbb{Z},
    \]
    where the sequence $a =(a_{1}, \dots, a_{n})$ is a \emph{Mal'cev basis}
    and the integers $f_{1}(g), \dots, f_{n}(g)$ are the \emph{Mal'cev
    coordinates} of $g$ relative to this basis.

    \begin{coro}
      Given a finitely generated group $\Gamma$ of polynomial growth,
      there exists a normal, finite-index, LO subgroup $H$. \label{coro_subgrupo_OE}
    \end{coro}

    \begin{proof}
      Let $\Gamma$ be a finitely generated group of polynomial growth. By
      Gromov's Theorem \cite{gromov1981groups}, $\Gamma$ has a nilpotent subgroup
      of finite index which, by Corollary \ref{grupo_nilpotente_torção}, contains
      a torsion-free, nilpotent normal subgroup $H \leq \Gamma$ of finite
      index.

      By Theorem \ref{teorema de Malcev}, there exists a Mal'cev basis $\{a
      _{1}, \dots, a_{n}\}$ for $H$. We define the positive cone
      $P \subset H$ as follows: an element $u \in H$, with $u \neq 1$, belongs
      to $P$ (denoted $u > 1$) if and only if, in its decomposition
      $u = \prod_{i=1}^{n}a_{i}^{u_i}$, the first non-vanishing coordinate
      $u_{i}$ is positive.

      Define the relation $u \leq v$ if and only if $u^{-1}v \in P \cup \{1
      \}$. We verify that this defines a left-invariant total order:
      \begin{enumerate}
        \item \textit{Reflexivity:} $u^{-1}u = 1$, hence $u \leq u$.

        \item \textit{Antisymmetry:} If $u \leq v$ and $v \leq u$, then $u
          ^{-1}v \in P \cup \{1\}$ and $v^{-1}u = (u^{-1}v)^{-1}\in P \cup
          \{1\}$. Due to the uniqueness of the Mal'cev basis, if the first
          non-vanishing coordinate of $x$ is positive, that of
          $x^{-1}$ is negative. Thus, the only possibility for both to
          be in $P \cup \{1\}$ is $u^{-1}v = 1$, implying $u=v$.

        \item \textit{Transitivity:} If $u \leq v$ and $v \leq w$, then $x
          = u^{-1}v \in P \cup \{1\}$ and $y = v^{-1}w \in P \cup \{1\}$.
          We wish to show $xy = u^{-1}w \in P \cup \{1\}$. The product
          of two elements whose first non-vanishing coordinates are positive
          results in an element with the same property (due to the
          structure of the central series where the quotients are
          abelian). Thus, $u \leq w$.

        \item \textit{Left-Invariance:}
          \begin{align*}
            u \leq v & \iff u^{-1}v \in P \cup \{1\} \iff u^{-1}w^{-1}wv \in P \cup \{1\} \\
                 & \iff (wu)^{-1}(wv) \in P \cup \{1\} \iff wu \leq wv.
          \end{align*}
      \end{enumerate}
      Therefore, $H$ is left-orderable.
    \end{proof}

    Hence, $\Gamma$ admits a finite-index LO subgroup $H$, and the result follows
    from Theorem \ref{TLC_func}.
  \end{proof}

  \section{Homology}
  \label{sec:homology}

  \begin{proof}[Proof of Theorem \ref{TLC_hom}]
    We start with a lemma on how cycles decompose across connected components.

    \begin{lemma}
      \label{lemma_connected_generator} If a disconnected cycle represents
      a nonzero homology class, it can be written as a sum
      $z=\sum_{i=1}^{k}z_{i}$ of chains with pairwise disjoint supports, each
      supported in a single connected component of $K$; each $z_{i}$ is a
      cycle; and at least one $z_{i}$ represents a nonzero homology class.
    \end{lemma}
    \begin{proof}
      Let $z = \sum_{j \in J}\rho_{j}\sigma_{j}$ be a cycle such that its homology
      class $[z] \neq 0$ in $H_{n}^{r}(K)$, where $K$ is a subgraph of $G$.
      Let $\{J_{i}\}_{i=1}^{k}$ be a partition of the index set $J$ such that
      each set of indices $J_{i}$ refers to simplices belonging to the same
      connected component of $K$. Define, for each $1\le i\le k$,
      \[
        z_{i}:= \sum_{j \in J_i}\rho_{j}\sigma_{j}.
      \]

      Suppose that
      \[
        \partial_{n}z_{l}= \sum_{j \in J_l}\rho_{j}\partial_{n}\sigma_{j}
        \neq 0
      \]
      for some $l$. Since $z$ is a cycle, we have
      \[
        \partial_{n}z = \sum_{i=1}^{k}\partial_{n}z_{i}= 0,
      \]
      which implies
      \[
        \sum_{i \neq l}\partial_{n}z_{i}= -\partial_{n}z_{l}\neq 0.
      \]
      This would imply the existence of an $(n-1)$-simplex in the support of
      $\partial_{n}z_{l}$ that also appears in the support of $\sum_{i
      \neq l}\partial_{n}z_{i}$. However, such a simplex would necessarily
      be a face of $n$-simplices belonging to distinct connected components,
      which contradicts the connectivity property of local simplicial
      complexes. Hence $\partial_{n}z_{i}=0$ for all $1\le i\le k$, so
      each $z_{i}$ is a cycle.

      Moreover, if every $z_{i}$ were trivial in homology (i.e.,
      $z_{i}\in \im \partial_{n+1}^{r}$ for all $i$), then their sum
      $z = \sum_{i=1}^{k}z_{i}$ would also be a boundary, contradicting the
      hypothesis that $[z] \neq 0$. Therefore, at least one $z_{i}$ must represent
      a nonzero homology class.
    \end{proof}

    Next, we bound the effect of a single-edge perturbation on the $n$-th Betti number.
    Fix an edge $e=\{u,v\}$ and set $\tilde G:=G\setminus\{e\}$. For an increasing
    exhaustive sequence $(A_{r})_{r\ge1}$, write $X_{r}:=\Delta(G[A_{r}])$
    and $\tilde X_{r}:=\Delta(\tilde G[A_{r}])$. We compare $X_{r}$ and $\tilde
    X_{r}$ via the long exact sequence of the pair $(X_{r},\tilde X_{r})$,
    which separates the change in $H_{n}$ into classes killed and classes
    created when $e$ is added. By monotonicity,
    $\tilde X_{r}\subseteq X_{r}$. We will show in Lemma~\ref{var_local_homologia_por_aresta}
    that the relative chain complex $S_{\bullet}(X_{r},\tilde X_{r})$ is generated
    inside a fixed neighborhood of $u$. This implies that its relative
    homology groups are eventually independent of $r$, and that the difference
    $\beta_{n}(X_{r})-\beta_{n}(\tilde X_{r})$ stabilizes as $r\to\infty$.

    \begin{lemma}
      \label{var_local_homologia_por_aresta} Let $G=(V,E)$ be a locally finite,
      quasi-transitive graph and let $\Delta$ be a local simplicial
      complex rule on $G$ with basic diameter $T$. Fix an edge $e=\{u,v\}\in
      E$ and set $\tilde G=(V,E\setminus\{e\})$. Let $(A_{r})_{r\ge1}$ be
      an increasing exhaustive sequence and write
      \[
        X_{r}:=\Delta(G[A_{r}]),\qquad \tilde X_{r}:=\Delta(\tilde G[A_{r}
        ]).
      \]
      Then there exist $r_{0}\in\mathbb{N}$ and
      $R_{\mathrm{stab}}\ge r_{0}$ such that:
      \begin{enumerate}
        \item the relative chain complex $S_{\bullet}(X_{r},\tilde X_{r})$
          is generated by simplices with vertices in $B_{T}(u)$ and the
          relative groups $H_{k}(X_{r},\tilde X_{r})$ are independent of
          $r$ for $r\ge r_{0}$;

        \item the difference $\beta_{n}(X_{r})-\beta_{n}(\tilde X_{r})$ is
          constant for all $r\ge R_{\mathrm{stab}}$.
      \end{enumerate}
      Moreover, there is a constant $C=C(n,\Delta)$ such that for all
      $r\ge r_{0}$,
      \[
        |\beta_{n}(X_{r})-\beta_{n}(\tilde X_{r})|\le C.
      \]
    \end{lemma}

    \begin{proof}
      Choose $r_{0}$ such that $B_{T}(u)\subseteq A_{r}$ for all $r\ge r_{0}$
      (possible since $(A_{r})$ is exhaustive), and fix $r\ge r_{0}$.

      \smallskip
      \noindent
      \textit{Step 1: the relative complex is local and stabilizes.} Let $\sigma
      \in X_{r}\setminus \tilde X_{r}$. Then $\sigma\in\Delta(G[A_{r}])$
      but $\sigma\notin\Delta(\tilde G[A_{r}])$. By the confinement property
      (basic diameter $T$), there exists a subgraph
      $\Lambda_{\sigma}\subseteq G[A_{r}]$ with
      \[
        \sup_{x,y\in V(\Lambda_\sigma)}d_{G}(x,y)\le T \quad\text{and}\quad
        \sigma\in\Delta(\Lambda_{\sigma}).
      \]
      If $e\notin E(\Lambda_{\sigma})$ then
      $\Lambda_{\sigma}\subseteq \tilde G[A_{r}]$ and, by monotonicity,
      $\sigma\in\Delta(\Lambda_{\sigma})\subseteq \Delta(\tilde G[A_{r}])=\tilde
      X_{r}$, a contradiction. Hence $e\in E(\Lambda_{\sigma})$. In particular,
      $u\in V(\Lambda_{\sigma})$, and for every $w\in V(\Lambda_{\sigma})$
      we have $d_{G}(w,u)\le T$. Therefore all vertices of $\sigma$ lie in
      $B_{T}(u)$.

      Thus every simplex representing a nonzero class in the quotient
      \[
        S_{k}(X_{r},\tilde X_{r}):=S_{k}(X_{r})/S_{k}(\tilde X_{r})
      \]
      has its vertices in $B_{T}(u)$. Equivalently,
      $S_{\bullet}(X_{r},\tilde X_{r})$ is generated by simplices
      contained in $B_{T}(u)$.

      Let $L_{r}$ and $\tilde L_{r}$ be the induced subcomplexes of
      $X_{r}$ and $\tilde X_{r}$ on the vertex set $B_{T}(u)$. The inclusion
      of pairs $(L_{r},\tilde L_{r})\hookrightarrow (X_{r},\tilde X_{r})$ induces
      an isomorphism of relative chain complexes, hence
      \[
        H_{k}(L_{r},\tilde L_{r})\cong H_{k}(X_{r},\tilde X_{r})\qquad\text{for
        all }k.
      \]
      Since $B_{T}(u)\subseteq A_{r}$ for all $r\ge r_{0}$, the induced subgraphs
      $G[A_{r}][B_{T}(u)]$ and $\tilde G[A_{r}][B_{T}(u)]$ do not depend on
      $r$ for $r\ge r_{0}$. By the locality of $\Delta$, the pairs $(L_{r},
      \tilde L_{r})$ are therefore constant for $r\ge r_{0}$, and so are
      the groups $H_{k}(X_{r},\tilde X_{r})$.

      \smallskip
      \noindent
      \textit{Step 2: eventual stabilization of the Betti variation.} Consider
      the long exact sequence of the pair $(X_{r},\tilde X_{r})$:
      \[
        \cdots \to H_{n+1}(X_{r},\tilde X_{r})\xrightarrow{\ \partial_{n+1}^r\ }
        H_{n}(\tilde X_{r}) \xrightarrow{\ i_*^r\ }H_{n}(X_{r})\xrightarrow
        {\ q_*^r\ }H_{n}(X_{r},\tilde X_{r}) \xrightarrow{\ \partial_n^r\ }
        H_{n-1}(\tilde X_{r})\to\cdots .
      \]
      Exactness gives
      \[
        \ker(i_{*}^{r})=\im(\partial_{n+1}^{r}), \qquad \coker(i_{*}^{r})
        \cong \im(q_{*}^{r})=\ker(\partial_{n}^{r})\subseteq H_{n}(X_{r},
        \tilde X_{r}),
      \]
      and hence, over $\R$,
      \[
        \beta_{n}(X_{r})-\beta_{n}(\tilde X_{r})=\dim\coker(i_{*}^{r})-\dim
        \ker(i_{*}^{r}).
      \]

      By Step~1, the relative group $H_{n}(X_{r},\tilde X_{r})$ is
      constant and finite-dimensional for $r\ge r_{0}$; fix identifications
      \[
        H_{n}(X_{r},\tilde X_{r})\cong V_{n},\qquad r\ge r_{0}.
      \]
      Let $Z_{r}:=\ker(\partial_{n}^{r})\subseteq V_{n}$. If
      $[x]\in Z_{r}$, then $\partial_{n}^{r}([x])=0$ in $H_{n-1}(\tilde X_{r}
      )$, i.e., a relative cycle representative $x$ has boundary
      $\partial_{n}x$ that bounds in $\tilde X_{r}$. Since $\tilde X_{r}\subseteq
      \tilde X_{r+1}$, the same bounding chain works in $\tilde X_{r+1}$, hence
      $[x]\in Z_{r+1}$. Therefore,
      \[
        Z_{r_0}\subseteq Z_{r_0+1}\subseteq Z_{r_0+2}\subseteq \cdots \subseteq
        V_{n}.
      \]
      Because $V_{n}$ is finite-dimensional, this ascending chain
      stabilizes: there exists $r_{1}\ge r_{0}$ such that $Z_{r}=Z_{r_1}$ for
      all $r\ge r_{1}$. In particular, $\dim\coker(i_{*}^{r})=\dim Z_{r}$ is
      constant for $r\ge r_{1}$.

      Similarly, Step~1 yields a fixed finite-dimensional space
      $V_{n+1}\cong H_{n+1}(X_{r},\tilde X_{r})$ for $r\ge r_{0}$. Set $J_{r}
      :=\ker(\partial_{n+1}^{r})\subseteq V_{n+1}$. If $[\xi]\in J_{r}$, then
      $\partial_{n+1}^{r}([\xi])=0$ in $H_{n}(\tilde X_{r})$, and by
      inclusion $\tilde X_{r}\subseteq \tilde X_{r+1}$ we also have
      $\partial_{n+1}^{r+1}([\xi])=0$, hence $J_{r}\subseteq J_{r+1}$.
      Thus $(J_{r})$ stabilizes at some $r_{2}\ge r_{0}$. Since
      \[
        \dim\ker(i_{*}^{r})=\dim\im(\partial_{n+1}^{r})=\dim V_{n+1}-\dim
        J_{r},
      \]
      it follows that $\dim\ker(i_{*}^{r})$ is constant for all $r\ge r_{2}$.
      With $R_{\mathrm{stab}}:=\max\{r_{1},r_{2}\}$, we conclude that $\beta
      _{n}(X_{r})-\beta_{n}(\tilde X_{r})$ is constant for all $r\ge R_{\mathrm{stab}}$.

      \smallskip
      \noindent
      \textit{Uniform bound.} For every $r\ge r_{0}$, we have $\dim\coker(i
      _{*}^{r})\le \dim V_{n}$ and $\dim\ker(i_{*}^{r})\le \dim V_{n+1}$,
      hence
      \[
        |\beta_{n}(X_{r})-\beta_{n}(\tilde X_{r})| \le \dim V_{n}+\dim V_{n+1}
        =:C,
      \]
      where $C$ depends only on $n$ and on the fixed local pair induced by
      $\Delta$ on $B_{T}(u)$.
    \end{proof}

    The remainder of the argument parallels Theorem~\ref{TLC_func},
    requiring only that we check the stability and moment conditions for
    $\beta_{n}$. By the hypotheses of Theorem \ref{TLC_hom}, let
    $H \leq \Gamma$ be the finite-index left-orderable subgroup. Let $S$ be the
    finite symmetric generating set of the Cayley graph, and let
    $X = \{x_{1}, \dots, x_{m}\}$ be a set of representatives of the right
    cosets of $H$ in $\Gamma$. We choose the fundamental set of edges
    \[
      \tilde{E}= \bigl\{\{x_{i},x_{i}s\}: x_{i}\in X,\, s \in S\bigr\}.
    \]
    Define the Betti functional $\beta_{n}:\Omega \times \mathcal{A}\to \mathbb{R}$
    by $\beta_{n}(\omega,A):=\dim H_{n}\bigl(\Delta(G(\omega;A))\bigr)$. To
    apply Theorem \ref{TLC_func}, it suffices to verify that $\beta_{n}$ is stationary,
    stabilizes in sequence, and satisfies the finite moment condition on $\tilde
    {E}$.

    For any $u \in \Gamma$ and $A \in \mathcal{A}$, the subgraph
    $G(\omega;A)$ is isomorphic to
    $G(\tilde{\varphi}_{u}(\omega);\varphi_{u}(A))$, so by the equivariance property
    of the associated simplicial complexes, the Betti functional preserves
    its value under the group action. Thus, the functional is stationary.

    Let us denote the effect of altering the state of edge $e$ on the Betti
    functional as:
    \[
      D_{n,e}(\omega, xA_{r}) = \beta_{n}(\omega, xA_{r}) - \beta_{n}(\omega
      ^{e}, xA_{r}),
    \]
    where the subscript $n$ will henceforth be omitted, and $xA_{r}$ denotes
    the translation $\varphi_{x}(A_{r})$.

    \begin{claim}
      Given $e \in \tilde{E}$, there exists a random variable
      $D_{e}^{\infty}$ such that, for any $x \in V$, the sequence
      $(D_{e}(xA_{r}))_{r\in\mathbb{N}}$ converges in probability to
      $D_{e}^{\infty}$.
    \end{claim}

    \begin{proof}
      Let $e=\{u,v\} \in \tilde{E}$ and $x \in V$. Consider the increasing
      exhaustive sequence $(xA_{r})_{r\ge1}$. By exhaustivity, there
      exists $r_{0}=r_{0}(x)$ such that $B_{T}(u)\subseteq xA_{r}$ for all
      $r\ge r_{0}$. For each fixed configuration $\omega$, apply Lemma~\ref{var_local_homologia_por_aresta}
      to the pair of graphs that differ by the single edge $e$ inside the
      induced subgraphs on $xA_{r}$. It follows that there exists $R_{\mathrm{stab}}
      =R_{\mathrm{stab}}(\omega,e,x)\ge r_{0}$ such that, for all
      $r\ge R_{\mathrm{stab}}$,
      \[
        D_{e}(\omega, xA_{r}) = \beta_{n}(\omega, xA_{r}) - \beta_{n}(\omega
        ^{e}, xA_{r})
      \]
      is constant. Hence $(D_{e}(xA_{r}))_{r\ge1}$ is eventually constant
      $\Prob_{p}$-a.s., and thus converges a.s. (in particular, in probability)
      to a limit random variable, denoted by $D_{e}^{\infty}$.

      Moreover, the stabilized value depends only on the restriction of
      $\omega$ to a fixed neighborhood of the edge $e$ (within the basic diameter
      $T$). In particular, once $B_{T}(u)\subseteq xA_{r}$ the value does
      not depend on the choice of $x$ or on further enlargements of
      $xA_{r}$, so the limit $D_{e}^{\infty}$ is independent of $x$.
    \end{proof}

    Finally, for any subgraph $G(\omega)$, the basic diameter of the
    simplicial complex rule $\Delta(G(\omega))$ is at most the basic
    diameter of $\Delta(G)$. Thus, by Lemma
    \ref{var_local_homologia_por_aresta}, the local sensitivity of the $n$-th
    Betti number to a single-edge perturbation is uniformly bounded by a constant.
    It follows immediately that there exists $\gamma > 2$ such that
    \[
      \max_{e \in \tilde{E}}\sup_{A\in \mathcal{A}}\E_{p}[|D_{e}(A)|^{\gamma}
      ] < \infty.
    \]

    Since $\beta_{n}$ is stationary, stabilizes in sequence, and satisfies the
    finite moment condition on $\tilde{E}$, Theorem~\ref{TLC_func} applies
    under an $H$-invariant edge ordering and filtration
    $(\mathcal{F}_{e})_{e\in E}$ as defined in Section~\ref{sec:functionals},
    yielding the desired result.
  \end{proof}

  \bibliographystyle{plain}
  \bibliography{referencias}

@article{zhang2001martingale,
  author    = {Zhang, Yu},
  title     = {A Martingale Approach in the Study of Percolation Clusters on the $\mathbb{Z}^d$ Lattice},
  journal   = {Journal of Theoretical Probability},
  volume    = {14},
  pages     = {165--187},
  year      = {2001},
  doi       = {10.1023/A:1007877216583},
  publisher = {Springer}
}

@article{mcleish1974dependent,
  author    = {McLeish, Donald L.},
  title     = {Dependent central limit theorems and invariance principles},
  journal   = {The Annals of Probability},
  volume    = {2},
  number    = {4},
  pages     = {620--628},
  year      = {1974},
  publisher = {Institute of Mathematical Statistics},
  doi={10.1214/aop/1176996608}
}

@article{CoxGrimmett1984,
  author  = {Cox, D. R. and Grimmett, G. R.},
  title   = {Central limit theorems for associated random variables and the percolation model},
  journal = {The Annals of Probability},
  volume  = {12},
  pages   = {514--528},
  year    = {1984},
  doi={10.1214/aop/1176993303}
}

@book{GrimmettBook,
  author    = {Grimmett, Geoffrey R.},
  title     = {Percolation},
  series    = {Grundlehren der mathematischen Wissenschaften},
  volume    = {321},
  publisher = {Springer},
  address   = {Berlin, Heidelberg},
  year      = {1999},
  edition   = {2}
}

@article{KestenZhang1990,
  author  = {Kesten, Harry and Zhang, Yu},
  title   = {The probability of a large finite cluster in supercritical {B}ernoulli percolation},
  journal = {The Annals of Probability},
  volume  = {18},
  pages   = {537--555},
  year    = {1990}
}

@article{KestenZhang1997,
  author  = {Kesten, Harry and Zhang, Yu},
  title   = {A central limit theorem for critical first-passage percolation in two dimensions},
  journal = {Probability Theory and Related Fields},
  volume  = {107},
  pages   = {137--160},
  year    = {1997}
}

@article{Penrose2001CLT,
author = {Mathew D. Penrose},
title = {{A Central Limit Theorem With Applications to Percolation, Epidemics and Boolean Models}},
volume = {29},
journal = {The Annals of Probability},
number = {4},
publisher = {Institute of Mathematical Statistics},
pages = {1515 -- 1546},
year = {2001},
doi = {10.1214/aop/1015345760},
URL = {https://doi.org/10.1214/aop/1015345760}
}

@book{lyons2017probability,
  author    = {Lyons, Russell and Peres, Yuval},
  title     = {Probability on Trees and Networks},
  series    = {Cambridge Series in Statistical and Probabilistic Mathematics},
  volume    = {42},
  publisher = {Cambridge University Press},
  year      = {2017}
}

@article{KM13,
  author  = {Kahle, Matthew and Meckes, Elizabeth},
  title   = {Limit theorems for {B}etti numbers of random simplicial complexes},
  journal = {Homology, Homotopy and Applications},
  volume  = {15},
  number  = {1},
  pages   = {343--374},
  year    = {2013}
}

@article{YSA17,
  author  = {Yogeshwaran, D. and Subag, Eliran and Adler, Robert J.},
  title   = {Random geometric complexes in the thermodynamic regime},
  journal = {Probability Theory and Related Fields},
  volume  = {167},
  number  = {1--2},
  pages   = {107--142},
  year    = {2017},
  doi     = {10.1007/s00440-015-0678-9}
}

@article{LRS19,
  author  = {Lachi{\`e}ze-Rey, Rapha{\"e}l and Schulte, Matthias and Yukich, J. E.},
  title   = {Normal approximation for stabilizing functionals in stochastic geometry},
  journal = {The Annals of Applied Probability},
  volume  = {29},
  number  = {2},
  pages   = {931--993},
  year    = {2019}
}

@article{HST18,
  author  = {Hiraoka, Yasuaki and Shirai, Tomoyuki and Trinh, Khanh Duy},
  title   = {Limit theorems for persistence diagrams},
  journal = {The Annals of Applied Probability},
  volume  = {28},
  number  = {5},
  pages   = {2740--2780},
  year    = {2018}
}

@article{KH22,
  author  = {Krebs, Johannes and Hirsch, Christian},
  title   = {Functional central limit theorems for persistent {B}etti numbers on cylindrical networks},
  journal = {Scandinavian Journal of Statistics},
  volume  = {49},
  number  = {1},
  pages   = {427--454},
  year    = {2022},
  doi     = {10.1111/sjos.12524}
}

@article{KP25,
  author  = {Krebs, Johannes and Polonik, Wolfgang},
  title   = {On the asymptotic normality of persistent {B}etti numbers},
  journal = {Advances in Applied Probability},
  volume  = {57},
  number  = {2},
  pages   = {492--523},
  year    = {2025},
  doi     = {10.1017/apr.2024.61}
}

@article{OS25,
  author    = {Owada, Takashi and Samorodnitsky, Gennady},
  title     = {Limit theorems for high-dimensional {B}etti numbers in the multiparameter random simplicial complexes},
  journal   = {Stochastic Processes and their Applications},
  volume    = {186},
  pages     = {104641},
  year      = {2025},
  doi       = {10.1016/j.spa.2025.104641},
  publisher = {Elsevier}
}

@article{OT21,
  author  = {Owada, Takashi and Thomas, Andrew M.},
  title   = {Limit theorems for process-level {B}etti numbers for sparse and critical regimes},
  journal = {Advances in Applied Probability},
  volume  = {52},
  number  = {1},
  pages   = {1--31},
  year    = {2020},
  doi     = {10.1017/apr.2019.50}
}

@article{SugimineTakei2006,
  author  = {Sugimine, Nobuaki and Takei, Masato},
  title   = {Remarks on {C}entral {L}imit {T}heorems for the {N}umber of {P}ercolation {C}lusters},
  journal = {Publications of the Research Institute for Mathematical Sciences},
  volume  = {42},
  number  = {1},
  pages   = {101--116},
  year    = {2006},
  doi = {10.2977/PRIMS/1166642060}
}

@article{PenrosePeres2011,
  author  = {Penrose, Mathew D. and Peres, Yuval},
  title   = {Local {C}entral {L}imit {T}heorems in {S}tochastic {G}eometry},
  journal = {Electronic Journal of Probability},
  volume  = {16},
  number  = {91},
  pages   = {2509--2544},
  year    = {2011},
  doi     = {10.1214/EJP.v16-968}
}

@book{manfred259ergodic,
  author    = {Einsiedler, Manfred and Ward, Thomas},
  title     = {Ergodic Theory with a View Towards Number Theory},
  series    = {Graduate Texts in Mathematics},
  volume    = {259},
  publisher = {Springer},
  year      = {2011}
}

@article{gromov1981groups,
  author  = {Gromov, Michael},
  title   = {Groups of polynomial growth and expanding maps (with an appendix by {J}acques {T}its)},
  journal = {Publications Math{\'e}matiques de l'IH{\'E}S},
  volume  = {53},
  pages   = {53--78},
  year    = {1981}
}

@book{baumslag1971lecture,
  author    = {Baumslag, Gilbert},
  title     = {Lecture Notes on Nilpotent Groups},
  series    = {CBMS Regional Conference Series in Mathematics},
  number    = {2},
  publisher = {American Mathematical Society},
  year      = {1971}
}

@book{kargapolov1979fundamentals,
  author    = {Kargapolov, Mikhail I. and Merzljakov, Jurij I.},
  title     = {Fundamentals of the Theory of Groups},
  series    = {Graduate Texts in Mathematics},
  volume    = {62},
  publisher = {Springer},
  year      = {1979}
}

@misc{garrido2013introduction,
  author       = {Garrido, Alejandra},
  title        = {An introduction to amenable groups},
  note = {Lecture notes},
  year         = {2013},
  url={https://www.math.uni-duesseldorf.de/~garrido/amenable.pdf}
}

@book{woess2000random,
  author    = {Woess, Wolfgang},
  title     = {Random Walks on Infinite Graphs and Groups},
  series    = {Cambridge Tracts in Mathematics},
  number    = {138},
  publisher = {Cambridge University Press},
  year      = {2000}
}

@article{grigorchuk1985degrees,
  author  = {Grigorchuk, Rostislav I.},
  title   = {Degrees of growth of finitely generated groups, and the theory of invariant means},
  journal = {Mathematics of the USSR-Izvestiya},
  volume  = {25},
  number  = {2},
  pages   = {259--300},
  year    = {1985}
}

@article{milnor1968,
  author    = {Milnor, John},
  title     = {Problem No. 5603},
  journal   = {The American Mathematical Monthly},
  volume    = {75},
  number    = {6},
  pages     = {685--687},
  year      = {1968},
  doi       = {10.2307/2313822},
  publisher = {Mathematical Association of America}
}

@misc{clay2023orderable,
  author       = {Clay, Adam},
  title        = {Orderable Groups Minicourse Notes},
    note = {Lecture notes},
  year         = {2023},
  url          = {https://adamjclay.github.io/montreal_notes.pdf}
}

@book{durrett2019probability,
  author    = {Durrett, Rick},
  title     = {Probability: Theory and Examples},
  edition   = {5},
  series    = {Cambridge Series in Statistical and Probabilistic Mathematics},
  volume    = {49},
  publisher = {Cambridge University Press},
  year      = {2019}
}

@article{BenjaminiSchramm1996,
  author  = {Benjamini, Itai and Schramm, Oded},
  title   = {Percolation beyond $\mathbb{Z}^d$, many questions and a few answers},
  journal = {Electronic Communications in Probability},
  volume  = {1},
  pages   = {71--82},
  year    = {1996}
}

@article{AVY18,
  author  = {Reddy, Agnishom and Vadlamani, Sunderesh and Yogeshwaran, D.},
  title   = {Central limit theorem for exponentially quasi-local statistics of spin models on {C}ayley graphs},
  journal = {Journal of Statistical Physics},
  volume  = {173},
  number  = {1},
  pages   = {18--36},
  year    = {2018},
  doi     = {10.1007/s10955-018-2114-0}
}

@article{Kahle2009,
  author  = {Kahle, Matthew},
  title   = {Topology of random clique complexes},
  journal = {Discrete Mathematics},
  volume  = {309},
  number  = {6},
  pages   = {1658--1671},
  year    = {2009},
  doi ={10.1016/j.disc.2008.02.037}
}

@article{BobrowskiKahle2018,
  author  = {Bobrowski, Omer and Kahle, Matthew},
  title   = {Topology of random geometric complexes: a survey},
  journal = {Journal of Applied and Computational Topology},
  volume  = {1},
  number  = {3--4},
  pages   = {331--364},
  year    = {2018},
  doi     = {10.1007/s41468-017-0010-0}
}
\end{document}